\documentclass{fundam}

\publyear{25}
\papernumber{5}
\volume{194}
\issue{1}
\theDOI{10.46298/fi.13487}

\versionForARXIV

\usepackage{hyperref}

\usepackage{amssymb}
\usepackage{amssymb}
\usepackage{latexsym}
\usepackage{amsfonts}
\usepackage{fancyhdr}
\usepackage{bm}
\usepackage[utf8]{inputenc}
\usepackage{calc}
\usepackage{geometry}
\usepackage{epstopdf}

\usepackage{subfigure}
\usepackage{amsmath}
\usepackage{times}
\usepackage{graphics}
\usepackage{float}
\usepackage[misc]{ifsym}
\usepackage{color}
\usepackage{threeparttable}
\usepackage{footnote}
\usepackage{geometry}

\usepackage{bm}
\usepackage{graphicx}
\usepackage{float}
\usepackage{bm}
\usepackage{graphicx}
\usepackage{footnote}
\usepackage[misc]{ifsym}
\usepackage{xcolor}
\usepackage{threeparttable}
\usepackage{footnote}
\usepackage{times}
\usepackage[marginal]{footmisc}
\usepackage{geometry}
\usepackage{float}

\usepackage{graphicx}
\usepackage{subfigure}
\usepackage{amsmath}
\usepackage{times}
\usepackage{threeparttable}
\usepackage{amssymb}
\usepackage{latexsym}
\usepackage{amsfonts}
\usepackage{fancyhdr}
\usepackage{bm}
\usepackage[utf8]{inputenc}
\usepackage{calc}
\usepackage{geometry}
\usepackage{epstopdf}

\begin{document}

\title{{A concentration phenomenon for $h$-extra edge-connectivity  reliability analysis of  enhanced hypercubes $Q_{n,2}$ with exponentially many faulty links}\thanks{This work was supported by the Basic scientific research in universities of Xinjiang Uygur Autonomous Region (Grant No. XJEDU2024P012), the National Natural Science Foundation of China (Grant Nos. 12561063 and 11961051), the Science and Technology Project of Xinjiang Uygur Autonomous Region (No. 2024D01C38), and the Doctoral Startup Foundation of Xinjiang University (Grant No. 62031224736).}}

\author{Yali Sun \\
College of Mathematics and System Sciences, Xinjiang University, Urumqi 830046, China \\
y\_lisun@163.com
\and
Mingzu Zhang\thanks{Corresponding author}\\
College of Mathematics and System Sciences, Xinjiang University, Urumqi 830046, China \\
mzuzhang@163.com
\and
Xing Feng \\
School of Science, Jimei University, Xiamen, Fujian 361021, PR China \\
fengxing\_fm@163.com
\and
Xing Yang \\
Department of Engineering Technologies, Monroe Community College, Rochester, NY 14623 \\
xyang4@monroecc.edu
}

\runninghead{Y. Sun, M. Zhang, X. Feng, X. Yang}{A concentration phenomenon for $h$-extra edge-connectivity\ldots}

\maketitle

\begin{abstract}
Reliability assessment of interconnection networks is critical to the design and maintenance of multiprocessor systems. The $(n, k)$-enhanced hypercube $Q_{n,k}$, as a variation of the hypercube $Q_{n}$, was proposed by Tzeng and Wei in 1991. As an extension of traditional edge-connectivity, $h$-extra edge-connectivity of a connected graph $G,$ $\lambda_h(G),$ is an essential parameter for evaluating the reliability of interconnection networks. This article intends to study the $h$-extra edge-connectivity of the $(n,2)$-enhanced hypercube $Q_{n,2}$. Suppose that the link malfunction of an interconnection network $Q_{n,2}$ does not isolate any subnetwork with no more than $h-1$ processors, the minimum number of these possible faulty links concentrates on a constant $2^{n-1}$ for each integer $\lceil\frac{11\times2^{n-1}}{48}\rceil \leq h \leq 2^{n-1}$ and $n\geq 9$. That is, for about $77.083\%$ of values where $h\leq2^{n-1},$ the corresponding $h$-extra edge-connectivity of $Q_{n,2}$, $\lambda_h(Q_{n,2})$, presents a concentration phenomenon. Moreover, the lower and upper bounds of $h$ mentioned above are both tight.
\end{abstract}




\section{Introduction}

\par\setlength{\parindent}{1em} The growing need to process and store massive amounts of data has led to increase more interest in multiprocessor systems. The advent of multiprocessor systems with a large number of processors and links meets this requirement \cite{LFCYJ2024, YZM2023, ZZ2024}. As the scale of such these systems continues to increase, so does the probability of links malfunctioning or failing. In addition, finding an appropriate parameter to measure the reliability of the system is crucial to the design and maintenance of the multiprocessor system \cite{LZHL2024, QMS2024, SFMC2024, YZM2024, YZM2025, ZZCH2024}. It is well known that the underlying topology of an interconnection network can be modeled by a connected graph $G=(V, E)$, with vertex set $V$ representing processors and edge (link) set $E$ representing the communication links between processors. The degree of a vertex in $G$ is the number of edges incident to it. A graph $G$ is called $d$-regular if and only if the degree of each vertex in graph $G$ is $d$.

 The performance of the interconnection network can usually be reflected by the topological parameters of its underlying connected graph $G$ \cite{GE2023}. The connectivity and edge-connectivity are two essential parameters for the reliability and fault tolerance assessment of interconnection networks. The connectivity $\kappa(G)$ or the edge-connectivity $\lambda(G)$ of a connected graph $G$ is defined as the minimum number of vertices or edges whose removal from $G$ makes the remaining graph disconnected. For most common graphs, their connectivity or edge connectivity frequently coincides exactly with their minimum degree. Moreover, the set of faulty edges that renders the graph disconnected is often linked to a unique vertex whose degree matches the minimum degree of the entire network. However, in the case of large, real-world networks, when disconnection occurs, it is rarely due to all faulty edges being concentrated around a single vertex. Consequently, while traditional definitions of connectivity or edge connectivity provide a foundation of a graph $G$, they may not offer a meticulous measure for assessing the fault tolerance capabilities of such networks. Therefore, Harary \cite{ref8} proposed more refined parameters, conditional connectivity and conditional edge-connectivity to meet this need in 1983. Due to the closed interconnection between various local parts of $G,$ when some malfunction of links and processors occurs, some parts of the local structures cannot be destroyed completely. The edges in a forbidden faulty edge set cannot fail simultaneously. By restricting the forbidden faulty edge set to the sets of neighboring edges of any induced subgraph with no more than $h-1$ vertices in the faulty networks, F\`{a}brega and Fiol \cite{ref13} proposed the $h$-extra edge-connectivity in 1996. Given a positive integer $h$, an $h$-extra edge-cut of a connected graph $G$ is defined as a set of edges of $G$ whose deletion yields a disconnected graph with all its components having at least $h$ vertices. The $h$-extra edge-connectivity of a connected graph $G$, denoted as $\lambda_{h}(G)$, is defined as the minimum cardinality of all $h$-extra edge-cuts of $G$.
Given a vertex set $X \subset V(G)$, the complement of a vertex set $X$ is $\overline{X}=V(G)\setminus {X}$. $G[X]$ and $[X,\overline X]_{G}$ can be defined as
the subgraph induced by the vertex set $X$ and the set of edges of $G$ in which each edge
contains one end vertex in $X$ and the other end vertex in $\overline{X}$, respectively.
Let $\xi_{m}(G)=\min \{|[X,\overline{X} ]_{G}| : | X \mid=m \leq\lfloor{\mid V(G)|}/2\rfloor$, $ G[X] $ is connected\}.
If $\lambda_{h}(G)=\xi_{h}(G)$, it is called $\lambda_{h}$-optimal; otherwise, it is not $\lambda_{h}$-optimal. Many authors studied
exact values of the $h$-extra edge-connectivity of some promising interconnection networks, such as hypercubes \cite{ref28}, folded hypercubes \cite{ref3,ref32,ref29,ref27,ref22,ref4,ref40}, $BC$ networks \cite{ref41,ref45,ref12,ref23,ref55}, and $3$-ary $n$-cubes \cite{MZMM2021}. The specific conclusions are shown in Table~\ref{tab1}.

\begin{table}[t]
\centering
\caption{Previous knowns and current results on the $h$-extra edge-connectivity for some classes of interconnection networks.}\label{tab1}
\resizebox{1 \columnwidth}{!}{
\begin{threeparttable}
\begin{tabular}{@{}llllllllllllllllllllll@{}}
\hline
\hline
Graph&$h$ &$\lambda_{h}$ &Authors  \\
\hline
$Q_{n}$  &$1\leq h\leq 2^{\lfloor\frac n2 \rfloor}$    &$nh-ex_h(Q_n), n\geq4 $  &Li and Yang \cite{ref28} in 2013 		\\
\hline
$Q_{n}^{3}$  &$3^{\lceil\frac{n}{2}\rceil+r}-\lfloor\frac{3^{2r+e+1}}{2}\rfloor\leq h\leq 3^{\lceil\frac {n}{2}\rceil+r}$    &$2(\lfloor\frac{n}{2}\rfloor-r)3^{\lceil\frac{n}{2}\rceil+r}, n\geq3 $  &Ma et al. \cite{MZMM2021} in 2021\\
\hline
 $FQ_{n}$ &$1$   &$n+1, n\geq2$   &El-Amawy and Latifi \cite{ref32} in 1991\\
          &$2$     &$2n, n\geq2$   &Zhu and Xu \cite{ref4} in 2006\\
          &$3$   &$3n-1, n\geq5$    &Zhu et al. \cite{ref40} in 2007\\
          &$4$    &$4n-4, n\geq5$   &Chang et al. \cite{ref3} in 2014\\
          &$\leq n$    &$\xi_{h}(FQ_{n}), n\geq 6$  &Yang and Li \cite{ref29} in 2014\\
 &$\leq2^{\lceil\frac{n}{2}\rceil+1}-4$, for odd $n$ &$\xi_h(FQ_n), n\geq4$ &Zhang et al. \cite{ref27} in 2016\\ &$\leq2^{\lceil\frac{n}{2}\rceil+1}-2$, for even $n$
 &$\xi_h(FQ_n), n\geq4$ &Zhang et al. \cite{ref27} in 2016			 			\\
       &$2^{\lceil\frac{n}{2}\rceil+1}-d_r\leq h\leq2^{\lceil\frac{n}{2}\rceil+1}$ \tnote{$\pounds$}
  &$\lfloor\frac{n}{2}\rfloor2^{\lceil\frac{n}{2}\rceil+1}, n\geq4$   &Zhang et al. \cite{ref27} in 2016\\

       &$2^{\lfloor\frac{n}{2}\rfloor+r}-l_r\leq h\leq2^{\lfloor\frac{n}{2}\rfloor+r}$  \tnote{$\ddag$} &$(\lceil\frac{n}{2}\rceil-r+1)2^{\lfloor\frac{n}{2}\rfloor+r}$       &Zhang et al. \cite{ref27} in 2016\\
       &$1\leq h \leq2^{n-1}$  &Algorithm  &Zhang et al. \cite{ref22} in 2018\\
       \hline
 $B_n$  &$1$    &$n$          &Chen et al. \cite{ref41} in 2003\\
 &$2$   &$2n-2$  &Chen et al. \cite{ref41} in 2003\\
 &$3$   &$3n-5$  &Zhu et al. \cite{ref55} in 2006\\
 &$4$   &$4n-8$   &Hong and  Hsieh  \cite{ref45} in 2013\\
 &$\frac{2^{n-1}+2^f}{3} \leq h\leq 2^{n-1}$  \tnote{$\S$}  &$2^{n-1}$ &Zhang et al. \cite{ref23} in 2014\\
 \hline
$Q_{n,k}$  &$1$   &$2n, 5\leq k\leq n-1$ &Sabir et al.  \cite{ref24} in 2019\\
     &$2$ &$3n-1, 5\leq k\leq n-1$ &Sabir et al.  \cite{ref24} in 2019\\
    &$1\leq h\leq2^{\lceil\frac{n}{2}\rceil}-d_{r}, n\leq2k+3, k\geq3$ \tnote{$\pounds$} &$(n+1)h-\sum_{i=0}^{s} t_{i} 2^{t_{i}}+\sum_{i=0}^{s} 2 i 2^{t_{i}}$ &Xu et al. \cite{ref36} in 2021\\
   &$2^{\lceil\frac{n}{2}\rceil}-d_{r}\leq h\leq 2^{\lceil\frac{n}{2}\rceil}, 2k\leq n \leq 2k+3$ &$(n+1)h-\sum_{i=0}^{s} t_{i} 2^{t_{i}}+\sum_{i=0}^{s} 2 i 2^{t_{i}}$ &Xu et al. \cite{ref36} in 2021\\
  &$2^{\lceil\frac{n}{2}\rceil}-d_{r}\leq h\leq 2^{\lceil\frac{n}{2}\rceil}, k+2\leq n \leq 2k-1$ &$\lfloor\frac{n}{2}\rfloor2^{\lceil\frac{n}{2}\rceil}$ &Xu et al. \cite{ref36} in 2021\\
 &$1\leq h\leq2^{\lceil\frac{n}{2}\rceil+1}-d_{r}, n\geq2k+4$ &$(n+1)h-\sum_{i=0}^{s} t_{i} 2^{t_{i}}+\sum_{i=0}^{s} 2 i 2^{t_{i}}$  &Xu et al. \cite{ref36} in 2021\\
 &$2^{\lceil\frac{n}{2}\rceil+1}-d_{r}\leq h \leq 2^{\lceil\frac{n}{2}\rceil+1}$ &$\lfloor\frac{n}{2}\rfloor2^{\lceil\frac{n}{2}\rceil+1}$ &Xu et al. \cite{ref36} in 2021\\
 \bm{$Q_{n,2}$}    &\bm{$\lceil\frac {11\times2^{n-1}}{48}\rceil\leq h \leq 2^{n-1}$}   &\bm{$2^{n-1}, (k=2)$}  &\textbf{Current}\\
\hline
\hline
\end{tabular}
\begin{tablenotes}
\footnotesize
\item[$\ddag$] where $r=1, 2, \ldots, \lceil\frac{n}{2}\rceil-1$ and $l_r=\frac{2^{2r-1}}{3}$ if $n$ is odd and $l_r=\frac{2^{2r-2}}{3}$ if $n$ is even.~~~~~~~~~~~~~~~~~~~
\item[$\S$] where $f=0$ if $n$ is even, and $f=1$ if $n$ is odd.~~~~~~~~~~~~~~~~~~~~~~~~~~~~~~~~~~~~~~~~~~~~~~~~~~~~~~~~~~~~~~~~~~~~~~~~~~~~
\item[$\pounds$] where $d_{r}=2$ if $n$ is even, and $d_{r}=4$ if $n$ is odd.~~~~~~~~~~~~~~~~~~~~~~~~~~~~~~~~~~~~~~~~~~~~~~~~~~~~~~~~~~~~~~~~~~~~~~~~
\end{tablenotes}
\end{threeparttable}}
\end{table}

The enhanced hypercube is a variant of the hypercube. Based on $n$-dimensional hypercube $Q_{n}$, Tzeng and Wei \cite{ref1} proposed the concept of $(n,k)$-enhanced hypercube $Q_{n,k}$ for $1\leq k\leq n-1,$ by adding different types of complementary edges. Compared to $Q_{n}$, the $(n,k)$-enhanced hypercube $Q_{n,k}$ performs very well in many measurements, such as notably reduced mean internode distance, considerably smaller diameter, highly optimized traffic density, robust connectivity, exceptional fault tolerance, remarkable cost-effectiveness \cite{ref1}, superior communication ability, and outstanding diagnosability \cite{ref33}. Undoubtedly, the enhanced hypercubes $Q_{n,k}$ require more hardware to build than hypercubes $Q_{n}$. However, when $n$ is large, the expense is minimal, and the benefits of the structural advantages are substantial. Due to
attractive properties, the $(n,k)$-enhanced hypercube has been widely studied.

Recently, the $h$-extra edge-connectivity and $h$-extra connectivity of $Q_{n,k}$ are widely investigated. For the edge version, Sabir et al. \cite{ref24} investigated $\lambda_h(Q_{n,k})$ for $h = 1, 2$ in 2019; while Xu et al. \cite{ref36} studied $\lambda_{h}(Q_{n,k})$ for $1 \leq h \leq 2^{\lceil\frac{n}{2}\rceil+1}, n \geq 2k + 4 $ and
$1 \leq h \leq 2^{\lceil\frac{n}{2}\rceil}, n \leq 2k + 3, k \geq 3$ in 2021. For the vertex
version, Li et al. \cite{ref25} determined that $\kappa_1(Q_{n,k})$ for $n = k + 1$ and $k \geq 1$, $\kappa_2(Q_{n,k})$ for $n = k + 1$ and $k \geq 3$ and $\kappa_3(Q_{n,k})$ for $n = k + 1$ and $k \geq 3$ in 2020. Sabir et al. \cite{ref24} also determined $\kappa_1(Q_{n,k})$ for $n \geq 7, 2 \leq k \leq n -5$ and $\kappa_2(Q_{n,k})$
 for $n \geq 9$ and $2 \leq k \leq n-7$ in 2019.
 Yin and Xu \cite{ref62} proved $\kappa_g(Q_{n,k})$ for $0\leq g\leq n-k-1, 4\leq k\leq n-5$ and $n\geq9$ in 2022. In particular, for $k=1$, the $(n,k)$-enhanced hypercubes $Q_{n,k}$ is $n$-dimensional folded hypercubes $FQ_{n}$. They also allow a linear number of malfunctions. It is not enough. Aim to go further, we consider the cases under the exponentially many faulty links.

In 2013, Li and Yang investigated $\lambda_h(Q_{n})$ for $1\leq h\leq 2^{\lfloor\frac{n}{2}\rfloor}$ and $n\geq 4$. In 2014, Yang and Li \cite{ref29} determined $\lambda_h(FQ_{n})$ for $h\leq n$ and $n\geq 6$. In 2014, Zhang et al. \cite{ref23} studied $\lambda_h(B_n)$ for $1 \leq h \leq 2^{\lfloor\frac{n}{2}\rfloor+1}$ and $n\geq 4$. In 2014, Yang and Meng \cite{ym2014} investigated $\kappa_g(Q_{n})$ for $0 \leq g \leq n-4$. In 2017, Zhou \cite{Z2017} determined $\kappa_g(HL_n)$  for $0 \leq g \leq n-3$ and $n \geq 5$. Compared to classical Menger theory, both $h$-extra edge-connectivity and $g$-extra connectivity offer a more refined measure of fault tolerance and reliability in interconnection networks. For very small $h$ or $g$, they usually satisfy the $\lambda_h$-optimality $\lambda_{h}(G)=\xi_h(G)$ or $\kappa_g$-optimality $\kappa_g(G)=\xi_g^v(G)$.

 If, for every integer $h_1 \leq h \leq h_2,$ the value of the parameter $\lambda_{h}(G)$ is a constant, then one says that the $h$-extra edge-connectivity of a graph $G$ is concentrated for the interval $h_1 \leq h \leq h_2,$ and this represents a concentration phenomenon. As large as possible $h_1 \leq h \leq h_2$ and $\lambda_{h_{1}-1}(G) < \lambda_{h_{1}}(G)=\lambda_{h}(G)=\lambda_{h_{2}}(G)< \lambda_{h_{2}+1}(G),$ it means that this interval $h_1 \leq h \leq h_2$ is maximal. In particular, for $h=h_1=h_2,$ $\lambda_{h}(G)=\xi_h(G)$ is $\lambda_{h}$-optimal. Recently, Zhang et al. (2016) \cite{ref27} studied the values of $\lambda_h(FQ_n)$, which concentrate on $\lfloor\frac{n}{2}\rfloor2^{\lceil\frac{n}{2}\rceil+1}$ for $2^{\lceil\frac{n}{2}\rceil+1}-d_{r}\leq h \leq 2^{\lceil\frac{n}{2}\rceil+1}$, where  $d_r= 4$ if $n$ is odd and $d_r = 2$ if $n$ is even. Xu et al. (2021) \cite{ref36} investigated the values of $\lambda_h(Q_{n,k})$ concentrate on $\lfloor\frac{n}{2}\rfloor2^{\lceil\frac{n}{2}\rceil}$ for $2^{\lceil\frac{n}{2}\rceil}-d_r\leq h\leq2^{\lceil\frac{n}{2}\rceil},$ where $d_r= 4$ if $n$ is odd and $d_r = 2$ if $n$ is even. With the increase of $r,$ the concentration phenomenon also becomes obvious. Zhang et al. (2016) \cite{ref27} also determined the values of $\lambda_h(FQ_n)$ concentrate on $(\lceil\frac{n}{2}\rceil-r+1)2^{\lfloor\frac{n}{2}\rfloor+r}$ for $2^{\lfloor\frac{n}{2}\rfloor+r}-l_r\leq h\leq2^{\lfloor\frac{n}{2}\rfloor+r},$ where $r=1,2, \cdots, \left\lceil\frac{n}{2}\right\rceil-1$ and $l_r=\frac{2^{2 r}-1}{3}$ if $n$ is odd and $l_r=\frac{2^{2 r+1}-2}{3}$ if $n$ is even. The study of the concentration phenomenon of $\lambda_ h(Q_{n,k})$ has just started. Inspired by the above results, this paper mainly focuses on the most obvious concentration phenomenon of $\lambda_h(Q_{n,2})$ in the subinterval $\lceil\frac{11\times2^{n-1}}{48}\rceil \leq h \leq 2^{n-1}$. For example, the values of $\xi_h(Q_{n,2})$
are marked in blue, the values of $\lambda_h(Q_{n,2})$ are marked in red, and the subinterval we examines is marked in
green (see Fig.~\ref{fig1}). Our main contributions are stated as follows.

\newpage

In the subsequent proof process, we will divide the subinterval $\lceil\frac{11\times2^{n-1}}{48}\rceil \leq h \leq 2^{n-1}$ into $\lceil\frac{n}{2}\rceil-1$ numbers of integer subintervals and we will define the $m_{n, r}$ as follows:

\begin{equation}
m_{n,r}=\left\{\begin{array}{ll}
\sum\nolimits_{i=0}^{2} 2^{n-4-i}+\sum\nolimits_{i=0}^{\lceil\frac{n}{2}\rceil-4-r 2^{n-8-2i}+2^{2r-1-f}}&\text { if } 1\leq r\leq\lceil\frac {n}{2}\rceil-4;\notag(e)\\
\sum\nolimits_{i=0}^{3} 2^{n-4-i} & \text { if } r=\lceil\frac{n}{2}\rceil-3;\notag(f)\\
2^{n-3} & \text { if } r=\lceil\frac{n}{2}\rceil-2;\notag(g)\\
2^{n-1} & \text { if } r=\lceil\frac{n}{2}\rceil-1,\notag(h)
\end{array}\right.\notag
\end{equation}
for $r=1, 2, \cdots, \lceil\frac{n}{2}\rceil-1.$
\begin{figure}[htbp]\label{fig1}
  \centering
 \includegraphics[height=8cm]{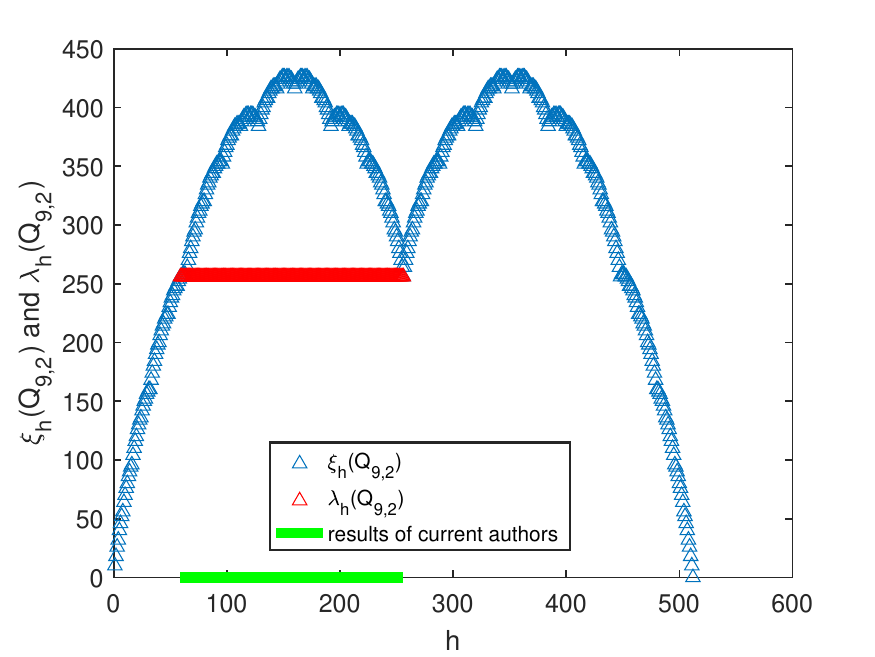}
  \caption{The values of $\xi_h(Q_{9,2})$ and $\lambda_h(Q_{9,2})$.}
\end{figure}

\begin{theorem}\label{c1}
For three integers $n \geq 9,$ $\lceil\frac{11\times2^{n-1}}{48}\rceil\leq h \leq 2^{n-1}$ and $1\leq r\leq\lceil\frac{n}{2}\rceil-1$, the results are as follows:

(a) $\lambda_h(Q_{n,2})=\xi_{2^{n-1}}(Q_{n,2})=2^{n-1}$;

(b) It is $\lambda_h$-optimal ($\lambda_h(Q_{n,2})=\xi_h(Q_{n,2})=2^{n-1}$) if and only if $h=m_{n,r}$.
\end{theorem}

 The rest of this paper is organized as follows.
 Section~\ref{sec2} introduces some related definitions and lemmas. Section~\ref{sec3} gives several lemmas about the properties of the function $\xi_{m}(Q_{n,2})$. Section~\ref{sec4} determines
that the value of the $h$-extra edge-connectivity of $Q_{n,2}$  concentrates on a constant $2^{n-1}$.
The last section concludes our results.
\section{ Preliminaries }\label{sec2}
Recall that the $h$-extra edge-connectivity of a connected graph $G$, $\lambda_{h}(G)$, is the minimum number of an edge-cut of the graph
$G$ whose removal separates the graph $G$ with all resulting components having at least $h$ vertices. If $F$ be a minimum $h$-extra edge-cut of a connected graph $G,$ then there is a fact that $G-F$ has exactly two components. In fact, if $F$ is the minimum $h$-extra edge-cut of the connected graph $G,$ $G-F$ has $p$ components $C_1, C_2, \cdots, C_p$ with at least $h$ vertices, $p\geq3$. Since the graph $G$ is  connected, there must exist integer $i, j$ with $[V_{C_i}, V_{C_j}]_{G}\neq \emptyset, |F_1|=|F\setminus[V_{C_i}, V_{C_j}]_{G}|<|F|$. Thus, $F_1$ is also a  minimum $h$-extra edge-cut of $G$, which contradicts the minimality of $F$. Hence, $G-F$ has exactly two components. Although the original definition of $\xi_m(G)$  only requires that $G[X]$ is connected, we do
need that both $G[X]$ and $G[\overline {X}]$ are connected in this paper. The function $\xi_m(Q_{n,2})$ of $(n,2)$-enhanced  hypercubes $Q_{n,2}$ does have the same result after modifying this condition. Let
\begin{align}
\xi_{m}(G)=\min\{|[X,\overline{X}]_{G}|: | X|=m\leq\lfloor |V(G)| /2\rfloor, \text{and} ~\text{both}~G[X]~ \text{and}~ G[\overline{X}]~\text{are}~ \text{connected}\}.
\end{align}
For some $d$-regular graphs,
\begin{align}
\xi_{m}(G)=d m-e x_{m}(G)\label{yy},
\end{align}
where $ex_{m}(G)$ is twice the maximum number of edges among all $m$ vertices induced subgraphs for each $m\leq \left\lfloor{\lvert V(G)\rvert}/2\right\rfloor$.
Actually, if we can find $X_{m}^{\ast}\subseteq V(G), \lvert X_{m}^{*} \rvert=m $, with $ex_{m}(G)=2 \lvert E(G[X_{m}^{*})] \lvert$, and so that, both $G[X_{m}^{*}]$ and $G[\overline {X_{m}^{*}}]$ are connected. Then
$$ \xi_{m}(G)=\lvert[X_{m}^{*}, \overline {X_{m}^{*}}]_{G}\rvert=dm-ex_m(G)=dm-2\lvert E(G[X_{m}^{*}]\rvert.$$
Specifically, for the hypercube, the folded hypercube and the $(n, 2)$-enhanced hypercube mentioned in this article, the conclusion $\xi_{m}(G)=dm-ex_{m}(G)$ is true. By the definition of the $h$-extra edge-connectivity of $G$, once the parameter $\lambda_{h}(G)$ is well-defined for $1\le h\le \lfloor|V(G)|/ 2\rfloor$, the $\xi_{m}(G)$ offers the upper bound for the $\lambda_h(G)$. Due to the fact that the smallest $h$-extra edge-cut of graph $G$ exactly divides graph $G$ into two connected components \cite{ref22}, so the parameter $\lambda_h(G)$ (by Zhang et al. \cite{ref22} page 299) can be calculated by
\begin{align}
\lambda_{h}(G)=\min \left\{\xi_{m}(G): h \leq m \leq\left\lfloor{\lvert V(G)\rvert}/2\right\rfloor\right\}.
\end{align}

Let $n, k$ be positive integers. The definitions of the $n$-dimensional hypercube
$Q_n,$ the folded hypercube $FQ_n$ and the $(n,k)$-enhanced hypercube $Q_{n,k}$ are stated as follows.

\begin{definition}\cite{s1977} For an integer $n \geq1,$ the $n$-dimensional hypercube, denoted by $Q_{n},$ is a graph with $2^{n}$ vertices. The vertex set $V(Q_n)=\{x_n x_{n-1} \cdots x_1: x_i \in\{0,1\}, 1 \leq i \leq n\}$ is the set of all $n$-bit binary strings. Two vertices  $x=x_{n} x_{n-1} \cdots x_{2}x_{1}$ and $y=y_{n} y_{n-1} \cdots y_{2}y_{1}$ of $Q_n$ are adjacent if and only if they differ in exactly one position.
\end{definition}

For any vertices $x=x_{n} x_{n-1} \cdots x_{2}x_{1}$ and $y=y_{n} y_{n-1} \cdots y_{2}y_{1},$ the edge $e=(x,y)$ is called $k$-complementary edge ($1\leq k\leq n-1$) if and only if $y_i=x_i$ for $n-k+1<i\leq n,$ and $y_j=\overline{x_j}$ for $1\leq j\leq n-k+1$.

As a variant of the hypercube, the $n$-dimensional folded hypercube $F Q_n$, first proposed by El-Amawy and Latifi \cite{ref32}, is a graph obtained from the hypercube $Q_n$ by adding an edge between every pair of vertices $x_n x_{n-1} \cdots x_1$ and $\overline{x_n}~ \overline{x_{n-1}} \cdots \overline{x_1}$, where $\overline{x_i}=1-x_i$ for all $1\leq i \leq n$. The $FQ_n$ is to add complementary edges between two $(n-1)$-dimensional sub-cubes. Motivated by this, by adding $k$-complementary edges between two paired $(n-k)$-dimensional sub-cubes, in 1991, Tzeng and Wei \cite{ref1} introduced the $(n,k)$-enhanced hypercube $Q_{n,k}$.
\begin{definition}
For two integers $n$ and $k$ with $1 \leq k \leq n-1$, the $(n, k)$-enhanced hypercube, denoted by $Q_{n, k},$ is defined to be a graph with the vertex set $V\left(Q_{n, k}\right)=\{x_{n}x_{n-1} \cdots x_{2}x_{1} : x_{i}\in \{0, 1\}, 1\leq i \leq n\}$. Two vertices $x=x_{n} x_{n-1} \cdots x_{2}x_{1}$ and $y=y_{n} y_{n-1} \cdots y_{2}y_{1}$ are adjacent if $y$ satisfies one of the following two conditions:

(1) $y=x_{n} x_{n-1} \cdots  {x}_{i+1}\bar{x}_{i}{x}_{i-1} \cdots x_{2}x_{1}$ for $1 \leq i \leq n,$ where $\bar{x}_{i}=1-x_{i}$ or

(2) $y=x_{n}x_{n-1} \cdots \bar{x}_{n-k+1} \bar{x}_{n-k} \cdots \bar{x}_{2}\bar{x}_{1}$.
\end{definition}

Note that $Q_{n,1}$ is the $n$-dimensional folded hypercube $FQ_{n}$. The $(n,2)$-enhanced hypercube $Q_{n,2}$ is obtained from the hypercube $Q_n$ by adding $2$-complementary edges between two pairs of vertices $x=x_{n} x_{n-1} \cdots x_{2}x_{1}$ and $y=x_{n}\bar x_{n-1} \cdots \bar{x}_{2}\bar{x}_{1}$ in two $(n-1)$-dimensional sub-cubes.

\begin{figure}[t]\label{fig2}
\centering
\begin{minipage}[t]{0.45\textwidth}
\includegraphics[height=4cm]{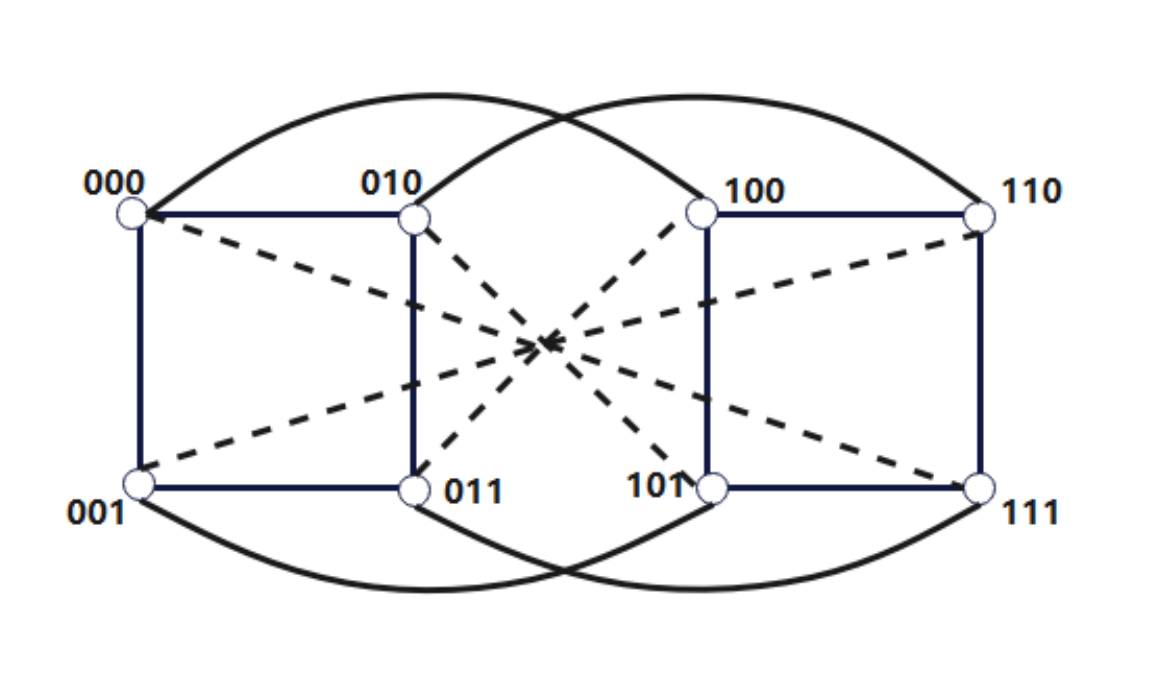}
\end{minipage}%
\begin{minipage}[t]{0.45\textwidth}
\centering
\includegraphics[height=4cm]{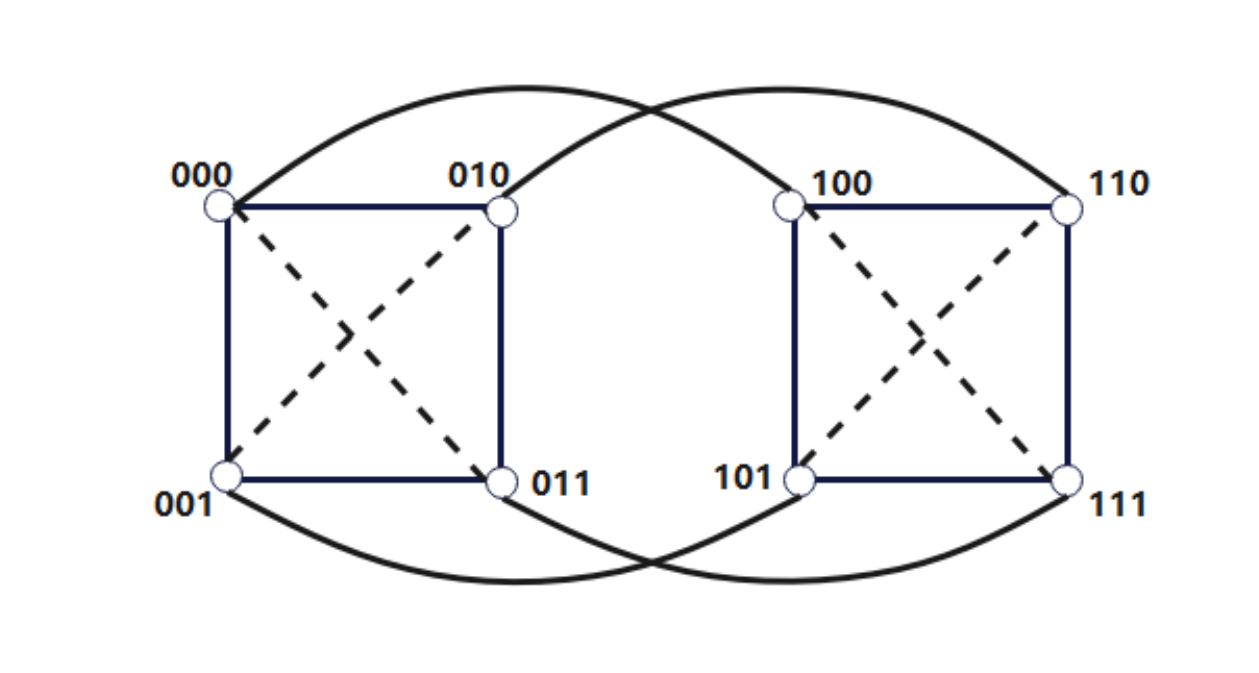}
\end{minipage}

\begin{minipage}[t]{0.33\textwidth}
\centering
\includegraphics[height=5cm]{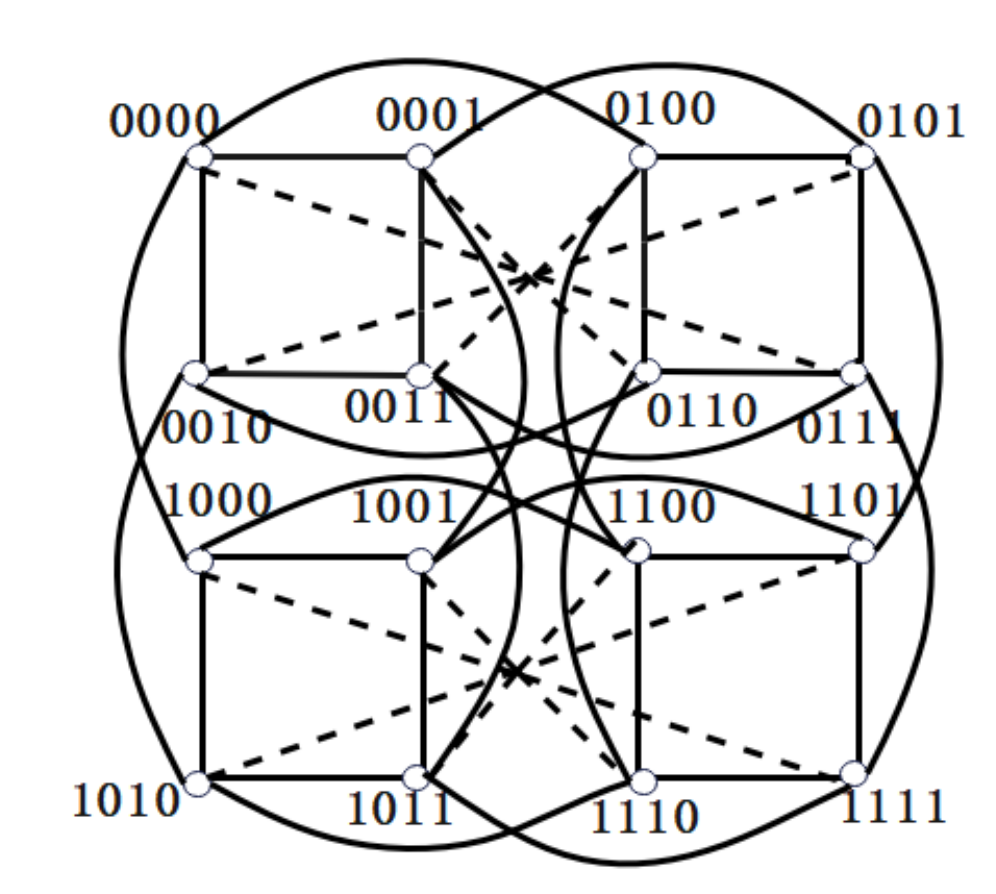}
\end{minipage}%
\caption{$Q_{3,1}$ (i.e. $FQ_3$), $Q_{3,2}$ and $Q_{4,2}$.}
\end{figure}

The $(n,2)$-enhanced hypercube $Q_{n,2}$ is $(n + 1)$-regular and $(n + 1)$-connected with $2^{n}$ vertices and $(n + 1)2^{n-1}$ edges \cite{ref1,ref33}.{The enhanced hypercubes $Q_{3,1}$, $Q_{3,2}$ and $Q_{4,2}$ are illustrated in Fig.~\ref{fig2}, where the complementary edges are represented by a short
dotted line. As the integer $n$ grows, the scale of $Q_{n,2}$  expands exponentially, and the topological structure of $Q_{n,2}$ becomes more and more complicated. Thus, the bitmaps of the adjacency matrix of $Q_{n,2}$  represent the adjacent relationship between vertices of $Q_{n,2}$. These figures of the adjacency matrix of $Q_{4,2}$, $Q_{5,2}$, $Q_{6,2}$ and $Q_{7,2}$ are shown in Fig.~\ref{fig3} (in four figures, the dark  pixel at location $(x, y)$ corresponds to the edges between vertices $x$ and $y$). The bitmaps of the adjacency matrix of $Q_{n,2}$ have high symmetry, iterative fractal, and sparsity.

\begin{figure}[htbp]
\centering
\begin{minipage}[t]{0.24\textwidth}
\vspace{4pt}
\centerline{
\includegraphics[width=4.5cm,height=4.5cm]{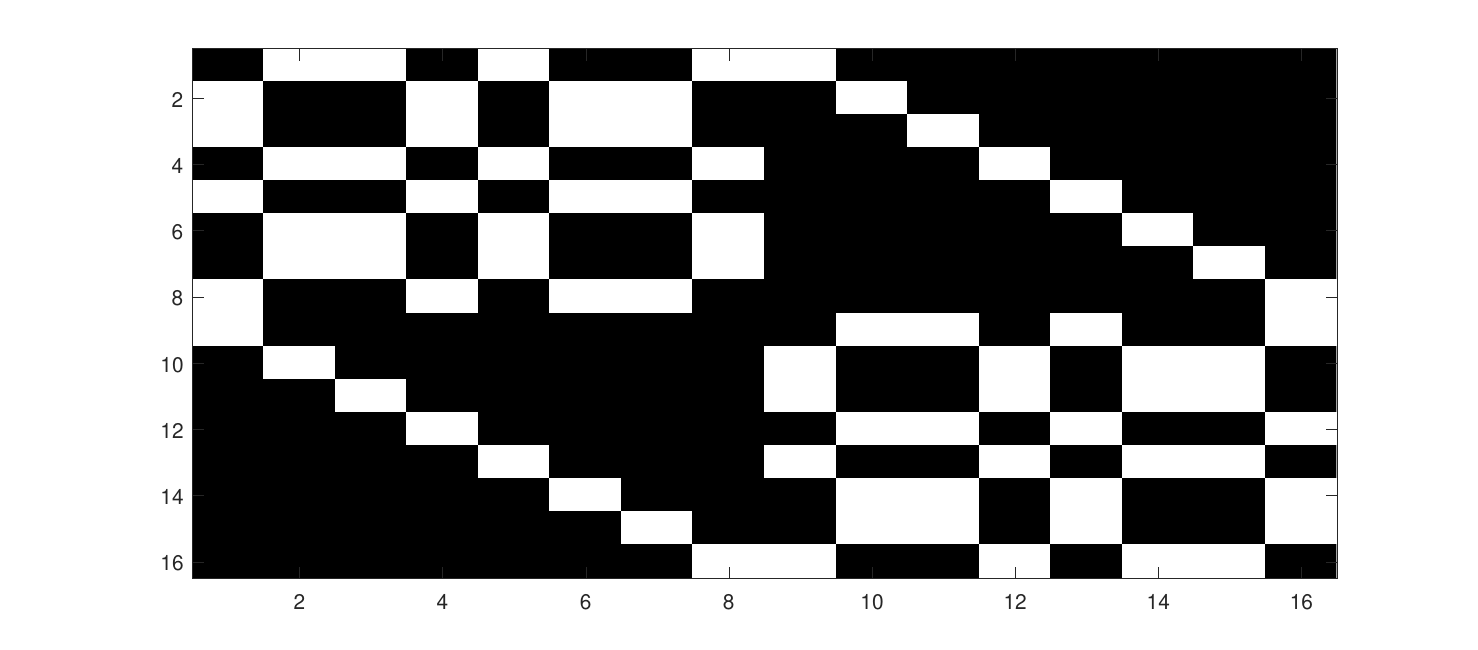}}
\end{minipage}
\begin{minipage}[t]{0.24\textwidth}
\vspace{4pt}
\centerline{
\includegraphics[width=4.5cm,height=4.5cm]{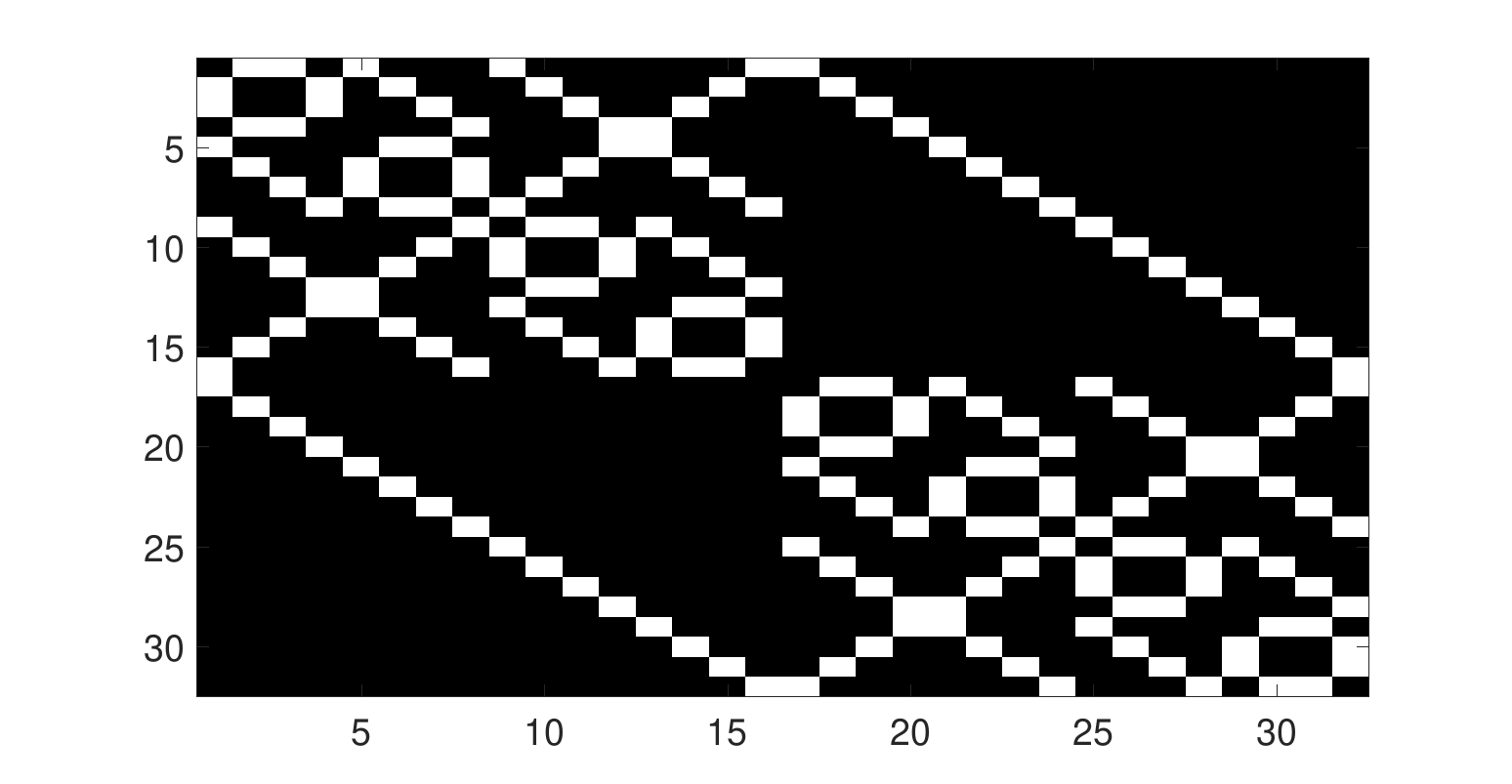}}
\end{minipage}
\begin{minipage}[t]{0.24\textwidth}
\vspace{4pt}
\centerline{
\includegraphics[width=4.5cm,height=4.5cm]{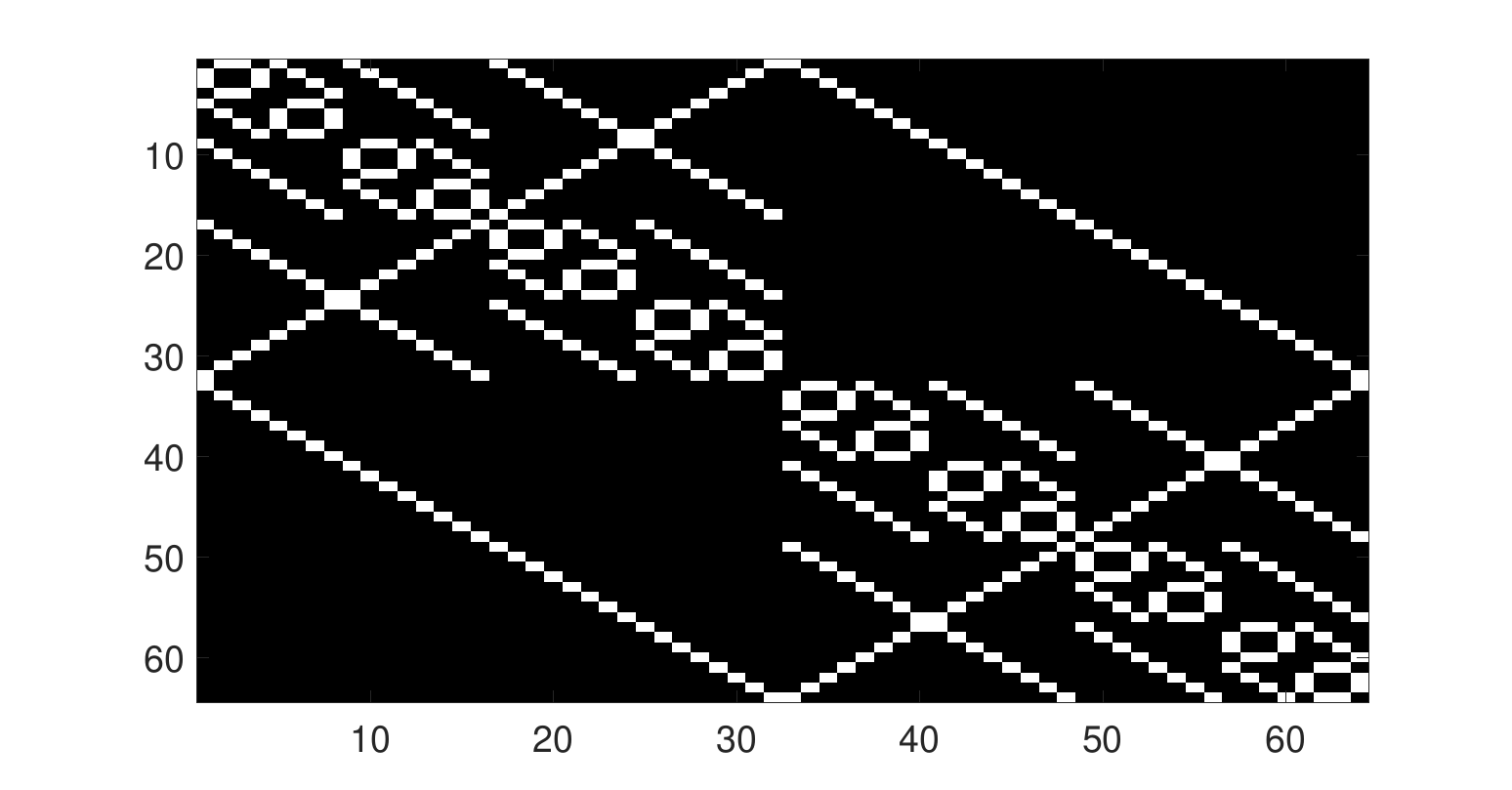}}
\end{minipage}
\begin{minipage}[t]{0.24\textwidth}
\vspace{4pt}
\centerline{
\includegraphics[width=4.5cm,height=4.5cm]{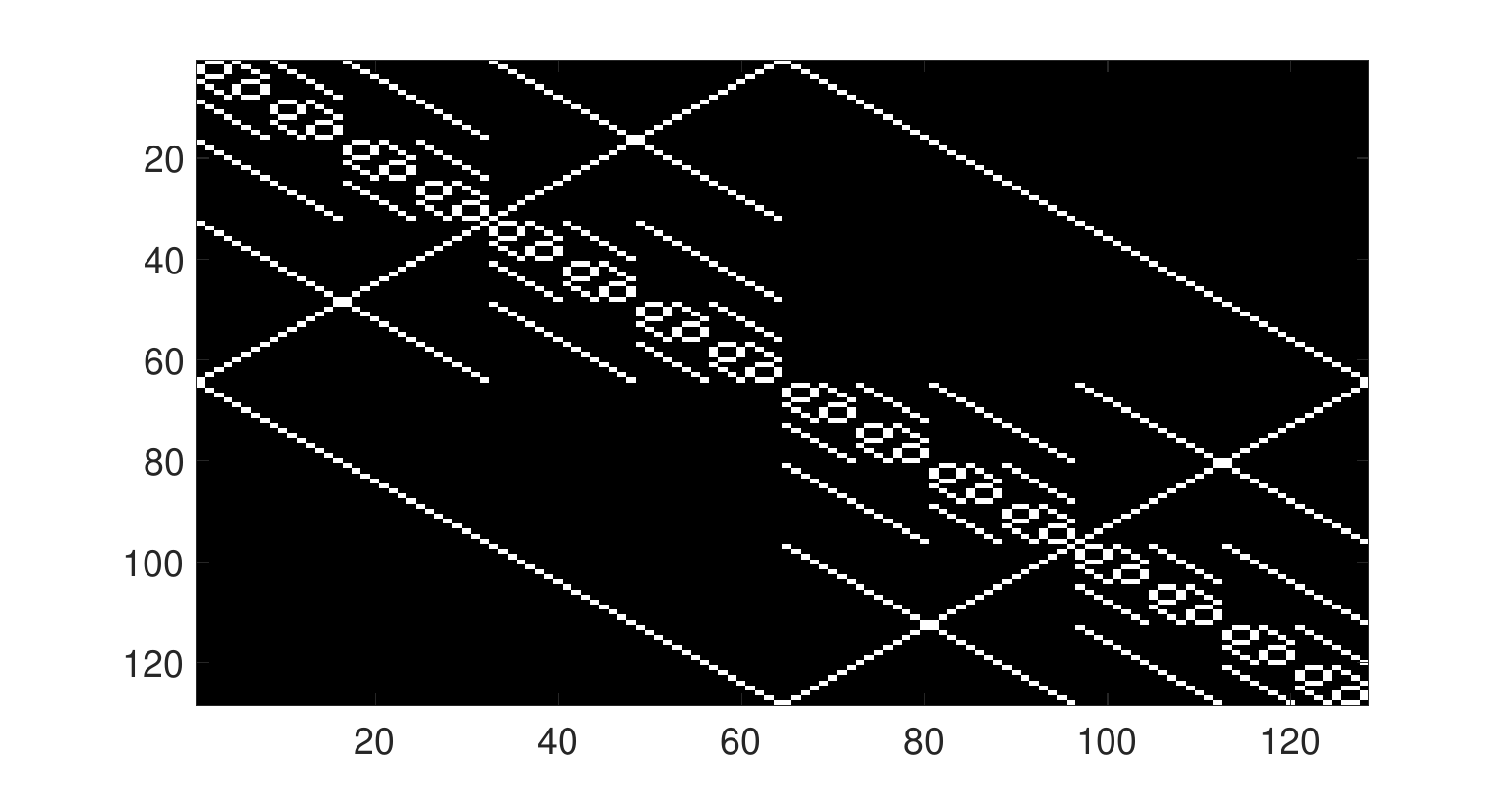}}
\end{minipage}
\caption{The bitmaps of adjacency matrix of $Q_{n,2}$ for $4 \le n \le 7$.}
\label{fig3}
\end{figure}

  For positive integers $1\leq m\leq 2^{n-1}$, there exists a unique binary representation $m=\sum_{i=0}^s 2^{t_i}$, where $t_0=\left\lfloor\log _2 m\right\rfloor$, $t_i=\left\lfloor\log _2\left(m-\sum_{r=0}^{i-1} 2^{t_r}\right)\right\rfloor$ for $i\geq1$, and $t_0>t_1>\cdots>t_{s}\ge0$. These conditions are used throughout the article when not causing ambiguity.
For every vertex $x=x_{n} x_{n-1} \cdots x_{1}$ of the $(n,2)$-enhanced hypercubes $Q_{n,2}$, it also can be denoted by decimal number $\sum_{i=1}^{n} x_{i} 2^{i-1}, x_{i}\in \{0,1\}$.  Let $S_m$ be the set $\{0,1,2, \cdots, m-1\}$ (under decimal representation). And $L_m^n$ denotes the corresponding set represented by $n$-binary strings. By the construction of $Q_{n,2},$ $L_m^n$ is a subset of $V\left(Q_{n,2}\right)$ and $Q_{n,2}\left[L_m^n\right]$ is a subgraph induced by $L_m^n$ in $Q_{n,2}$. Li and Yang proved that both $Q_{n}\left[L_m^n\right]$ and $Q_{n}\left[\overline{L_m^n}\right]$ are connected \cite{ref28} and the fact that $Q_{n}\left[L_m^n\right]$ and $Q_{n}\left[\overline{L_m^n}\right]$ are subgraphs of $ Q_{n,2}\left[L_m^n\right]$ and $ Q_{n,2}\left[\overline{L_m^n}\right]$, respectively, so both $ Q_{n,2}\left[L_m^n\right]$ and $ Q_{n,2}\left[\overline{L_m^n}\right]$ are connected. The subgraphs induced by $L_{m}^{4}$ in $Q_{n,2}$ for $m=4, 6$ and 8 are shown in Fig.~\ref{figure4}.

\medskip

 Happer \cite{h1964}, Li and Yang \cite{ref28} independently obtained the exact expression of the function $ex_{m}(Q_{n})$.
\begin{lemma}\label{s1}
(1) \cite{ref28} For a positive integer $m=\sum_{i=0}^{s} 2^{t_{i}}\leq2^{n}$, $\xi_m(Q_n)=nm-ex_{m}(Q_{n}),$ where $ex_{m}(Q_{n})=2\lvert E(Q_{n}[L_{m}^{n}])\rvert=\sum_{i=0}^{s} t_{i} 2^{t_{i}}+\sum_{i=0}^{s} 2 i 2^{t_{i}}$.
(2) \cite{ref28} $ex_{m_{0}+m_{1}}(Q_{n})\ge ex_{m_{0}}(Q_{n})+ex_{m_{1}}(Q_{n})+2m_{0}$ for positive integers $m_{0}\le m_{1}$.
\end{lemma}
Arockiaraj et al. \cite{ref31} obtained the exact expression of the function $ex_{m}(Q_{n,k})$ in 2019, which was rewritten by Xu et al. in 2021 \cite{ref36}.

\medskip

In the following, we let $[x]^{+}= x$ for $x \geq 0$; otherwise, $[x]^{+}=0$.

\newpage

\begin{lemma}\label{s3} \cite{ref31, ref36}
 For each integer $1 \leq m \leq 2^{n}$ and $ m=\sum_{i=0}^{s}2^{t_{i}},$ $\xi_m(Q_{n,2})=(n+1)m-ex_{m}(Q_{n,2}),$ where
\begin{equation}
\begin{array}{ll}
ex_m\left(Q_{n,2}\right)&=2\left\lvert E(Q_{n,2}[L_{m}^{n}])\right\rvert \\
&=2\left\lvert E\left(Q_{n}[L_{m}^{n}]\right)\right\rvert+ \lfloor\frac{m}{2^{n-1}}\rfloor2^{n-1}+2\left[m-\lfloor\frac{m}{2^{n-1}}\rfloor2^{n-1}-2^{n-2}\right]^{+}\\
&=\left\{\begin{array}{ll}
ex_{m}(Q_{n}) & \text { if } 1 \leq m \leq 2^{n-2};\\
ex_{m}(Q_{n})+2m-2^{n-1} & \text { if } 2^{n-2}<m \leq 2^{n-1}; \\
ex_{m}(Q_{n})+2^{n-1} & \text { if } m>2^{n-1} \text { and } m=2^{n-1}+x,\\
                                                                      &0 \leq x<2^{n-2};\\
ex_{m}(Q_{n})+2x & \text { if } m>2^{n-1} \text { and } m=2^{n-1}+x,\\
                                                                                &2^{n-2}\leq x<2^{n-1}.
\end{array}\right.
\end{array}
\end{equation}
\end{lemma}

Then several specific examples are used to illustrate the calculation of $ex_m(Q_{n,2})$. For example, for $n=4$ and $m=4$, $ex_{m}(Q_{n,2})=2|E(Q_{n,2}[L_{m}^{n}])|=\sum_{i=0}^{s} t_{i} 2^{t_{i}}+\sum_{i=0}^{s} 2 i 2^{t_{i}}$. Note that $S_{4}=\{0, 1,2, 3\}$ and $L_{4}^{4}=\{0000, 0001, 0010, 0011\}$. Since $4=2^{2}$, it can be seen that $t_{0}=2$ and {$e x_{4}(Q_{4,2})=2\lvert E(Q_{n,2}[L_{4}^{4}])\rvert=2 \times 2^{2}+2 \times 0 \times 2^{2}=8$; for $n=4$ and $m=8$, $ex_{m}(Q_{n,2})=\sum_{i=0}^{s} t_{i} 2^{t_{i}}+\sum_{i=0}^{s} 2 i 2^{t_{i}}+2m-2^{n-1}$. There are $S_{8}=\{0, 1, \ldots, 7\}$ and $L_{8}^{4}=\{0000, 0001, 0010, 0011, 0100, 0101, 0110, 0111\}$. Since $8=2^{3}$, it can be  obtained that $t_{0}=3$ and $e x_{8}(Q_{4,2})=2\lvert E(Q_{n,2}[L_{8}^{4}])\rvert=3 \times 2^{3}+2 \times 0 \times 2^{3}+2 \times 8-2^{3}=32$. The induced subgraphs $Q_{4,2}[L_{4}^{4}]$ and $Q_{4,2}[L_{8}^{4}]$ are shown in Fig.~\ref{figure4}.

\begin{figure}[t]
\centering
\begin{minipage}[t]{0.4\textwidth}
\vspace{11pt}
\centerline{
\includegraphics[height=3.5cm]{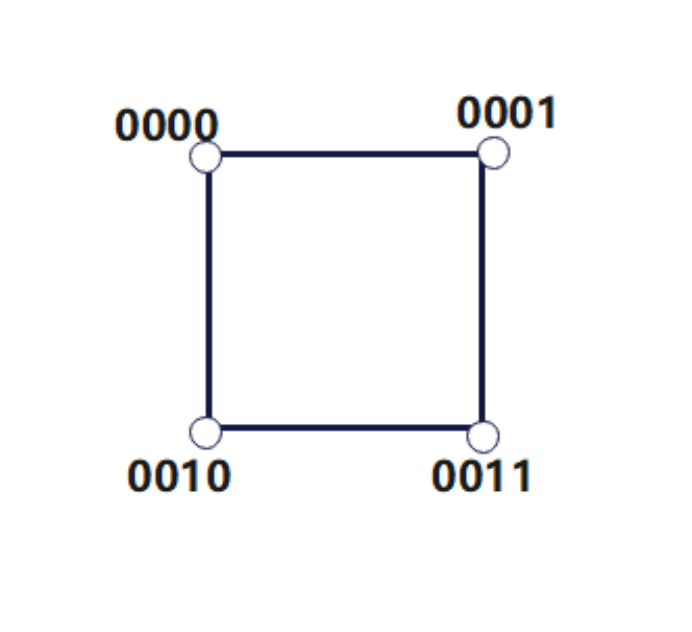}}
\end{minipage}
\begin{minipage}[t]{0.4\textwidth}
\vspace{4pt}
\centerline{
\includegraphics[height=3.5cm]{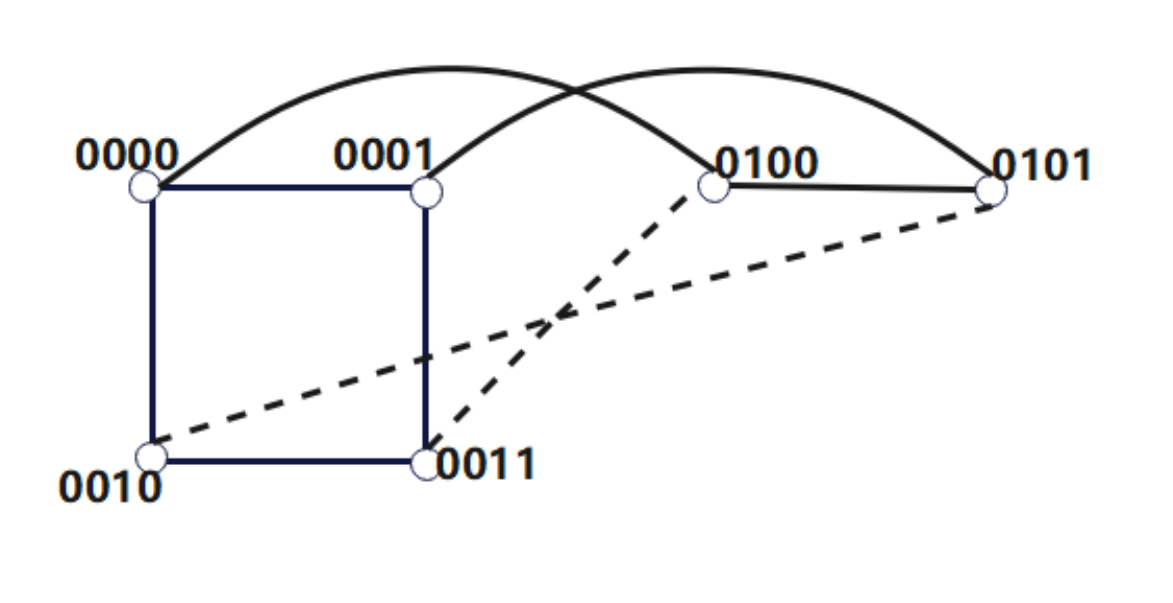}}
\end{minipage}
\begin{minipage}[t]{0.4\textwidth}
\vspace{4pt}
\centerline{
\includegraphics[height=3.5cm]{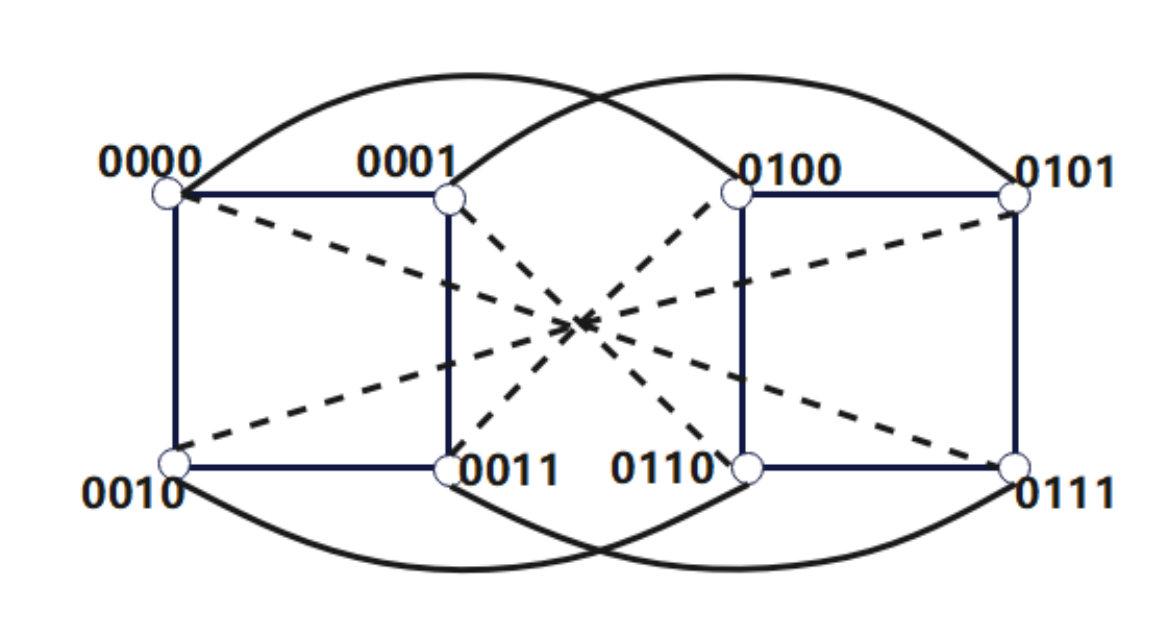}}
\end{minipage}
\caption{Induced subgraphs $Q_{4,2}[L_{4}^{4}]$, $Q_{4,2}[L_{6}^{4}]$ and $Q_{4,2}[L_{8}^{4}]$.}
\label{figure4}
\end{figure}

\begin{lemma}\label{s20}(\cite{ref36})
For positive integers $1 \leq m \leq 2^{t}$ and $ 0\leq t\leq n$, $ex_{m}(Q_{n}) \leq tm$ and $ex_{m}(Q_{n,k}) \leq (t + 1)m$.
\end{lemma}
\begin{lemma}\label{s5}(\cite{ref36})
For positive integers $h \leq m =\sum_{i=0}^{s}2^{t_{i}}\leq 2^{n-1},$
$$\lambda_{h}(Q_{n,2})=\min \left\{\xi_{m}(Q_{n,2}): h \leq m \leq 2^{n-1}\right\},$$
satisfying that
\begin{equation}
\begin{aligned}\label{formula5}
\xi_{m}(Q_{n,2})=(n+1) m-ex_{m}(Q_{n,2}).
\end{aligned}
\end{equation}
\end{lemma}
For $m \leq 2^{n-1},$ the following two iterative properties of the expression of $ex_m(Q_{n,2})$ depend on whether $Q_{n,2}$ matches complementary edges in the sub-network and the number of such complementary edges.
\begin{lemma}\label{s8}
Let $m, n$ be two integers, $n\geq 4, 1\leq m=\sum_{i=0}^{s}2^{t_{i}}\leq 2^{n-1}$. For $m_1=\sum_{i=0}^{a}2^{t_{i}},$  $m=m_{1}+m_{2}$, and $t_0>t_1 \cdot\cdot\cdot >t_a>t_{a+1}>t_{a+2}> \cdot\cdot\cdot >t_s\ge0, a < s$,
\begin{itemize}
\item[(a)] $ex_{m}(Q_{n,2})=ex_{m_{1}}(Q_{n,2})+e x_{m_{2}}(Q_{n,2})+2(a+1) m_{2}$ ~for ~$1 \leq m \leq 2^{n-2}$;
\item[(b)] $ex_{m}(Q_{n,2})=e x_{m_{1}}(Q_{n,2})+e x_{m_{2}}(Q_{n,2})+2(a+2)m_{2}$ ~for ~$2^{n-2}<m\leq 2^{n-1}$.
\end{itemize}
\end{lemma}
\begin{proof}
Note that $m_{2}=m-m_{1}=2^{t_{a+1}}+2^{t_{a+2}}+\cdots+2^{t_s}=\sum_{i=a+1}^{s} 2^{t_{i}}=\sum\nolimits_{i=0}^{s-a-1} 2^{t_{i+a+1}}$. Since the expression of $ex_m(Q_{n,2})$ strongly depends on the binary decomposition of $m$ and the domain of $m$, it can be divided into the following two cases according to its two different forms.
\begin{itemize}
\item[(a)] For $1 \leq m \leq 2^{n-2}$, by Lemma~\ref{s3}, it can be obtained
$$\textstyle{ex_{m_{1}}(Q_{n,2})=\sum_{i=0}^{a}t_{i}2^{t_{i}}+\sum\nolimits_{i=0}^{a}2i2^{t_{i}}}$$ and
$$\textstyle{ex_{m_{2}}(Q_{n,2})=\sum_{i=0}^{s-a-1}t_{i+a+1}2^{t_{i+a+1}}+\sum\nolimits_{i=0}^{s-a-1}2i2^{t_{i+a+1}}}.$$
Note that
\begin{align*}
 e x_{m}(Q_{n,2}) &=\textstyle{\sum\nolimits_{i=0}^{s} t_{i} 2^{t_{i}}+\sum\nolimits_{i=0}^{s} 2 i 2^{t_{i}}} \\
&=\textstyle{(\sum\nolimits_{i=0}^{a} t_{i} 2^{t_{i}}+\sum\nolimits_{i=0}^{s-a-1}t_{i+a+1} 2^{t_{i+a+1}})}\\
&\hspace{3cm}\textstyle{+(\sum\nolimits_{i=0}^{a} 2 i 2^{t_{i}}+\sum\nolimits_{i=0}^{s-a-1} 2(a+1+i)2^{t_{i+a+1}})} \\
&=\textstyle{e x_{m_{1}}(Q_{n,2})+e x_{m_{2}}(Q_{n,2})+2\sum\nolimits_{i=0}^{s-a-1}(a+1)2^{t_{i+a+1}}} \\
&=e x_{m_{1}}(Q_{n,2})+e x_{m_{2}}(Q_{n,2})+2(a+1)m_{2}.
\end{align*}
\item[(b)] For $2^{n-2}< m \leq 2^{n-1}$, by  Lemma~\ref{s3}, it is sufficient to show that
$$\textstyle{ex_{m_{1}}(Q_{n,2})=\sum_{i=0}^{a}t_{i}2^{t_{i}}+\sum\nolimits_{i=0}^{a}2i2^{t_{i}}+2m_1-2^{n-1}}$$ and
$$\textstyle{ex_{m_{2}}(Q_{n,2})=\sum_{i=0}^{s-a-1}t_{i+a+1}2^{t_{i+a+1}}+\sum\nolimits_{i=0}^{s-a-1}2i2^{t_{i+a+1}}}.$$
Note that
\begin{align*}
ex_{m}(Q_{n,2}) &= \textstyle{\sum\nolimits_{i=0}^{s} t_{i} 2^{t_{i}}+\sum\nolimits_{i=0}^{s}2 i 2^{t_{i}}+2m-2^{n-1}}\\
&=\textstyle{(\sum\nolimits_{i=0}^{a} t_{i}2^{t_{i}}+\sum\nolimits_{i=0}^{s-a-1}t_{i+a+1}2^{t_{i+a+1}})} \\
&\hspace{3cm}\textstyle{+(\sum\nolimits_{i=0}^{a} 2 i 2^{t_{i}}+\sum\nolimits_{i=0}^{s-a-1} 2(a+1+i)2^{t_{i+a+1}})+2m-2^{n-1}} \\
&=\textstyle{e x_{m_{1}}(Q_{n,2})+e x_{m_{2}}(Q_{n,2})+2m_{2}+2\sum\nolimits_{i=0}^{s-a-1}(a+1)2^{t_{i+a+1}}}\\
&=e x_{m_{1}}(Q_{n,2})+e x_{m_{2}}(Q_{n,2})+2(a+2)m_{2}.
\end{align*}
\end{itemize}
To sum up, the proof is completed.
\end{proof}
\section{Some properties of the function $\xi_{m}(Q_{n,2})$}\label{sec3}
The exact value of the function $\lambda_h(Q_{n,2})$ highly depends on the monotonic intervals and fractal structure of the function $\xi_m(Q_{n,2})$. Then we introduce several lemmas to describe the properties of the function $\xi_m(Q_{n,2})$.

Let $f=0$ if $n$ is even, and $f=1$ if $n$ is odd.
To deal with the interval $\left\lceil\frac{11\times2^{n-1}}{48}\right\rceil\leq m \leq 2^{n-1}$, by inserting $\lceil\frac{n}{2}\rceil-1$ numbers of $m_{n, r}$ satisfying
$$\textstyle{\left\lceil\frac{11\times2^{n-1}}{48}\right\rceil=m_{n,1}<m_{n,2}<\cdot\cdot\cdot<m_{n,r}<m_{n,r+1}<\cdot\cdot\cdot<m_{n,\lceil\frac{n}{2}\rceil-1}=2^{n-1}.}$$
  This interval is divided into $\lceil\frac{n}{2}\rceil-1$ numbers of integer subintervals. The expression of $m_{n, r}$ is defined as follows:
\begin{small}
\begin{equation}
m_{n,r}=\left\{\begin{array}{ll}
\sum\nolimits_{i=0}^{2} 2^{n-4-i}+\sum\nolimits_{i=0}^{\lceil\frac{n}{2}\rceil-4-r} 2^{n-8-2i}+2^{2r-1-f}&\text { if } 1\leq r\leq\lceil\frac {n}{2}\rceil-4;\notag(e)\\
\sum\nolimits_{i=0}^{3} 2^{n-4-i} & \text { if } r=\lceil\frac{n}{2}\rceil-3;\notag(f)\\
2^{n-3} & \text { if } r=\lceil\frac{n}{2}\rceil-2;\notag(g)\\
2^{n-1} & \text { if } r=\lceil\frac{n}{2}\rceil-1,\notag(h)
\end{array}\right.\notag
\end{equation}
\end{small}
for $r=1, 2, \cdots, \lceil\frac{n}{2}\rceil-1.$ By calculation, it can be obtained that
$$\textstyle{\left\lceil\frac{11\times2^{n-1}}{48}\right\rceil=m_{n,1}=\sum\nolimits_{i=0}^{2} 2^{n-4-i}+\sum\nolimits_{i=0}^{\lceil\frac{n}{2}\rceil-5} 2^{n-8-2i}+2^{1-f}}.$$
Actually, if $1 \leq r \leq \lceil\frac{n}{2}\rceil-4$ and $n$ is even, $m_{n,r}=\sum\nolimits_{i=0}^{2} 2^{n-4-i}+\sum\nolimits_{i=0}^{\lceil\frac{n}{2}\rceil-4-r} 2^{n-8-2i}+2^{2r-1}$.  $m_{n, 1}=2^{n-4}+2^{n-5}+2^{n-6}+2^{n-8}+2^{n-10}+\cdots+2^{2}+2^{1}$ and $3m_{n,1}=2m_{n,1}+m_{n,1}=2^{n-3}+2^{n-4}+2^{n-5}+2^{n-6}+2^{n-7}+\cdots+2^{3}+2^{2}+2^{1}+(2^{n-4}+2^{n-5}+2^{2})$, so $m_{n, 1}=\frac{11\times 2^{n-5}+2^{1-f}}{3}=\lceil\frac {11\times2^{n-1}}{48}\rceil$. If $n$ is odd,  $m_{n,r}=\sum\nolimits_{i=0}^{2} 2^{n-4-i}+\sum\nolimits_{i=0}^{\lceil\frac{n}{2}\rceil-4-r} 2^{n-8-2i}+2^{2r-2}$. $m_{n, 1}=2^{n-4}+2^{n-5}+2^{n-6}+2^{n-8}+2^{n-10}+\cdots+2^{1}+2^{0}$ and $3m_{n,1}=m_{n, 1}+2m_{n, 1}=2^{n-3}+2^{n-4}+2^{n-5}+2^{n-6}+2^{n-7}+\cdots+2^{2}+2^{1}++2^{0}+(2^{n-4}+2^{n-5}+2^{1})$, thus $m_{n, 1}=\frac{11\times 2^{n-5}+2^{1-f}}{3}=\lceil\frac {11\times2^{n-1}}{48}\rceil$.

Since, in the small-scale cases for $4\le n\le 8$, not all four anticipated scenarios occur (as detailed in Table~\ref{tab2}), this paper focuses primarily on the cases for $n\ge9$, and provides examples of the variables $r$ and $m_{n,r}$ for $n=9$ and $n=10$ (see Table~\ref{tab3}).

\begin{table}[t]
\caption{The variables of $r$, and $m_{n,r}$ for $4\leq n \leq8$.}\label{tab2}
\centering
\begin{tabular}{lllllllll}
\hline
$n$     &$m_{n,1}$   &$m_{n,2}$   &$m_{n,3}$ &$\cdot\cdot\cdot$   &$m_{n,\lceil\frac{n}{2}\rceil-1}$   \\
\hline
$4$     &$1, (h)$          &          &    &  &  & \\[-1pt]
$5$    &$4, (g)$  & $16, (h)$  &     &           &     \\[-1pt]
$6$     &$8, (g)$  &$32, (h)$ &  &     &    \\[-1pt]
$7$     &$15, (g)$  &$16, (h)$   &$64, (f)$     &           &    \\[-1pt]
$8$    &$30, (g)$ &$32, (h)$    &$128, (f)$  &  &\\[-1pt]
\hline
\end{tabular}
\end{table}

\begin{table}[h]
\caption{The variables of $r$, and $m_{n,r}$ for $n = 9$ or $10.$}\label{tab3}
\centering
\begin{tabular}{cccccccc}
\hline
&$n=9$ &                                         & $n=10$ &         &                         \\[-1pt] \hline
$r$   & $m_{n,r}$                       &     &$ m_{n,r}$                  \\[-1pt] \hline
$1$        & $59=2^{5}+2^{4}+2^{3}+2^{1}+2^{0}$ &         & $118=2^{6}+2^{5}+2^{4}+2^{2}+2^{1}$ \\[-1pt]
$2$        & $60=2^{5}+2^{4}+2^{3}+2^{2}$       &        & $120=2^{6}+2^{5}+2^{4}+2^{3}$ \\[-1pt]
$3$         & $64=2^{6}$                         &         & $128=2^{7}$                   \\[-1pt]
$4$        & $256=2^{8} $                        &           & $512=2^{9}$                   \\[-1pt] \hline
\end{tabular}
\end{table}

\begin{lemma}\label{s15} \cite{ref36} Let $c, n$ and $m$ be three integers, $n\geq 4, 0 \leq c \leq n-2$ and $2^{c}\leq m \leq 2^{n-1}$. Then $\xi_{m}(Q_{n,2}) \geq \xi_{2^{c}}(Q_{n,2})$.
\end{lemma}
\begin{lemma}\label{s9}
Let $n, r$ be two integers, $n\geq9,$ $r=1, 2, \cdots, \lceil\frac{n}{2}\rceil-1$. Then $\xi_{m_{n,r}}(Q_{n,2})=2^{n-1}$.
\end{lemma}
\begin{proof}
According to different expressions of $m_{n,r}$, the proof will be divided into four cases.

\noindent
\textbf{Case 1}. For $ 1 \leq r\leq\lceil\frac{n}{2}\rceil-4$, $m_{n,r}=\sum\nolimits_{i=0}^{2} 2^{n-4-i}+\sum\nolimits_{i=0}^{\lceil\frac{n}{2}\rceil-4-r} 2^{n-8-2i}+2^{2r-1-f},$ by Lemma~\ref{s3} and formula~\eqref{formula5}, it can be obtained that
\begin{align*}
\xi_{m_{n,r}}(Q_{n,2}) &= (n+1)m_{n, r}-e x_{m_{n, r}}(Q_{n,2}) \\
&=\textstyle{(n+1)\left[\sum\nolimits_{i=0}^{2} 2^{n-4-i}+\sum\nolimits_{i=0}^{\lceil\frac{n}{2}\rceil-4-r} 2^{n-8-2 i}+2^{2 r-1-f}\right]}\\
&\hspace{2cm}\textstyle{-\Big\{\sum\nolimits_{i=0}^{ 2 }\left[(n-4-i)2^{n-4-i}+2 i 2^{n-4-i})\right]} \\
&\hspace{2.5cm}+\textstyle{\sum\nolimits_{i=0}^{\lceil\frac{n}{2}\rceil-4-r}\left[(n-8-2 i) 2^{n-8-2 i}+2(3+i) 2^{n-8-2 i}\right]}\\
&\hspace{2.5cm}+\textstyle{\left[(2r-1-f)2^{2 r-1-f}+2(\lceil\frac{n}{2}\rceil-r) 2^{2 r-1-f}\right]}\Big\} \\
&=\textstyle{(n+1-n+4-i)\sum\nolimits_{i=0}^{2}2^{n-4-i}}\\
&\hspace{2cm}\textstyle{+(n+1-n+2)\sum\nolimits_{i=0}^{\lceil\frac{n}{2}\rceil-4-r}2^{n-8-2i}+(n+2+f-2\lceil\frac{n}{2}\rceil) 2^{2r-1-f}} \\
&=\textstyle{(5-i) \sum\nolimits_{i=0}^{2} 2^{n-4-i}+3 \sum\nolimits_{i=0}^{\lceil \frac{n}{2}\rceil-4-r} 2^{n-8-2i}+2^{2r-f}}\\
&=5 \cdot 2^{n-4}+4 \cdot 2^{n-5}+4 \cdot2^{n-6}-2^{2 r-f}+2^{2 r-f} =3 \cdot 2^{n-3}+2^{n-3} =2^{n-1}.
\end{align*}

\noindent
\textbf{Case 2}. For $ r=\lceil\frac{n}{2}\rceil-3, m_{n,r}=\sum\nolimits_{i=0}^{3} 2^{n-4-i}$, by Lemma~\ref{s3} and the formula~\eqref{formula5},
\begin{align*}
\xi_{m_{n, r}}(Q_{n,2})&=(n+1) m_{n, r}-e x_{m_{n, r}}(Q_{n,2})\\
&=\textstyle{(n+1)\sum_{i=0}^{3}(n-4-i)2^{n-4-i}-\sum\nolimits_{i=0}^{3}(n-4-i) 2^{n-4-i}-\sum\nolimits_{i=0}^{3}
2i2^{n-4-i}}\\
&=\textstyle{(5-i) \sum\nolimits_{i=0}^{3} 2^{n-4-i}=5\times2^{n-4}+4\times2^{n-5}+3\times2^{n-6}+2\times2^{n-7}=2^{n-1}}.
\end{align*}

\noindent
\textbf{Case 3}. For $ r=\lceil\frac{n}{2}\rceil-2, m_{n,r}=2^{n-3},$ by Lemma~\ref{s3} and the formula~\eqref{formula5}, it is not difficult to see
that $$\xi_{ 2^{n-3}}(Q_{n,2})=(n+1)\times2^{n-3}-(n-3)\times 2^{n-3}=2^{n-1}.$$

\noindent
\textbf{Case 4}. For $ r=\lceil\frac{n}{2}\rceil-1, m_{n,r}=2^{n-1},$ by the formula~\eqref{formula5} and  Lemma~\ref{s3}, then $$\xi_{2^{n-1}}(Q_{n,2})=( n+1 )\times 2^{n-1}-[(n-1)\times2^{n-1}+2\times2^{n-1}-2^{n-1}]=2^{n-1}.$$
From the above four cases, it can conclude that $\xi_{m_{n,r}}(Q_{n,2})=2^{n-1}$ for $r =1, 2, \textcolor{purple}{\cdots}, \lceil\frac{n}{2}\rceil-1$. The proof is completed.
\end{proof}
\begin{lemma}\label{s10}
Given two integers $n\geq9, \left\lceil\frac{11\times2^{n-1}}{48}\right\rceil\le m\leq2^{n-1},$ there exists a positive integer $r,$ satisfying $m_{n,r}<m< m_{n,r+1}$.
\begin{align*}
\xi_{m}\left(Q_{n,2}\right)>\xi_{m_{n,r}}\left(Q_{n,2}\right) &=\xi_{m_{n,r+1}}\left(Q_{n,2}\right)=\cdots=\xi_{m_{n,\lceil\frac{n}{2}\rceil-1}}\left(Q_{n,2}\right)\\
&=\xi_{2^{n-1}}(Q_{n,2})=\xi_{\left\lceil\frac{11\times2^{n-1}}{48}\right\rceil}(Q_{n,2})=2^{n-1}.
\end{align*}
\end{lemma}
\begin{proof}
According to different expressions of $ex_m(Q_{n,2})$, the proof will be divided into two cases.

{\noindent\bf Case 1.} $\left\lceil\frac{11\times2^{n-1}}{48}\right\rceil< m\leq2^{n-2}$.

One can check that $m_{n, r+1}-m_{n, r}= 2^{2r-1-f}$ for $1\leq r\leq \lceil\frac{n}{2}\rceil-3$. By Lemma~\ref{s9},
$$\xi_{m_{n,r}}(Q_{n,2})=\xi_{2^{n-1}}(Q_{n,2})=2^{n-1}$$
for $1 \leq r \leq \lceil\frac{n}{2}\rceil-1$. Let $m=m_{n, r}+p$, where
\begin{align*}
m_{n, r}&=\textstyle{\sum\nolimits_{i=0}^{2} 2^{n-4-i}+\sum\nolimits_{i=0}^{\lceil\frac{n}{2}\rceil-4-r} 2^{n-8-2i}+2^{2r-1-f}}&\textstyle{\text{for $1\leq r<\lceil\frac {n}{2}\rceil-4$}}, \\
m_{n,r}&=\textstyle{\sum\nolimits_{i=0}^{3} 2^{n-4-i}}&\textstyle{\text{for $r=\lceil\frac {n}{2}\rceil-3$}},
\end{align*}
$0 \leq p<2^{2r-1-f}, p=\sum_{i=0}^s2^{t'_{i}}<m_{n, r+1}-m_{n, r}, 2r-1-f>t'_{0}>t'_{1}>\cdot\cdot\cdot>t'_{s}$. By the equation~\eqref{formula5} and Lemma~\ref{s8}, one can deduce that
\begin{align*}
&\xi_{m}(Q_{n,2})-\xi_{m_{n,r}}(Q_{n,2}) \\
&\quad=(n+1)(m_{n,r}+p)-e x_{m_{n,r}+p}(Q_{n,2})-(n+1)m_{n,r}+e x_{m_{n,r}}(Q_{n,2})\\
&\quad=(n+1)p-e x_{m_{n, r}+p}(Q_{n,2})+e x_{m_{n, r}}(Q_{n,2})\quad\textrm{(Lemma~\ref{s3})}\\
&\quad=\textstyle{(n+1)p-e x_{m_{n, r}}(Q_{n,2})-ex_{p}(Q_{n,2})-2(\lceil \frac{n}{2}\rceil-r+1)p+e x_{m_{n, r}}(Q_{n,2})}\quad\textrm{(Lemma~\ref{s1}(2))}\\
&\quad=(2r-1-f)p-ex_p(Q_{n,2}) =(2r-1-f)p-ex_p(Q_{2r-1-f}) =\xi_p(Q_{2r-1-f}).
\end{align*}

For $p<2^{2r-1-f}<2^{n-1},$ the
value of $ex_{p}(Q_{n})$ is uniquely determined by the binary representation of $p$. Therefore, $ex_{p}(Q_{n})=ex_{p}(Q_{2r-1-f})$. By Lemma~\ref{s1},
$$ex_{p}(Q_{2r-1-f})=2\lvert E(Q_{2r-1-f}[L_{p}^{2r-1-f}])\rvert.$$
$[L_{p}^{2r-1-f},\overline{L_{p}^{2r-1-f}}]_{Q_{2r-1-f}}$ be an
edge cut of $Q_{2r-1-f}$. Since $Q_{2r-1-f}$ is connected graph, and if
one deletes the edge cut $[L_{p}^{2r-1-f}, \overline{L_{p}^{2r-1-f}}]_{Q_{2r-1-f}},$ two induced subgraphs $Q_{2r-1-f}[L_{p}^{2r-1-f}]$ and $Q_{2r-1-f}[\overline{L_{p}^{2r-1-f}}]$
are connected. So the edge cut $[L_{p}^{2r-1-f}, \overline{L_{p}^{2r-1-f}}]_{Q_{2r-1-f}}$ of $Q_{2r-1-f}$ does exist. By Lemma~\ref{s20}, it is sufficient to show that $ex_{p}(Q_{2r-1-f})\leq (2r-1-f)p,$ and $\xi_{m}(Q_{n,2})-\xi_{m_{n,r}}(Q_{n,2})=(2r-1-f)p-e x_{p}\left(Q_{2r-1-f}\right)>0$.

If $r=\lceil\frac{n}{2}\rceil-2$, then $m_{n,r}=2^{n-3}$. There exists a
positive integer $p'=\sum_{i=0}^{s}2^{t'_{i}},$ satisfying $0 \leq p'< 2^{n-3}$, $m=2^{n-3}+p'$ and $n-3>t'_{0}>t'_{1}>\cdot\cdot\cdot>t'_{s}$. The proof of $\xi_{m}(Q_{n,2})>\xi_{m_{n,r}}(Q_{n,2})$ is the same as the above proof of $\xi_{m}(Q_{n,2})>\xi_{m_{n,r}}(Q_{n,2}).$

\noindent
\textbf{Case 2}. $2^{n-2}< m\leq 2^{n-1}$.

If $r=\lceil\frac{n}{2}\rceil-2$, then $m_{n,r}=2^{n-3}$. There exists a
positive integer $m''=\sum_{i=0}^{s}2^{t'_{i}},$ satisfying $0 \leq m''< 2^{n-2}$, $m=m_{n,r}+2^{n-3}+m''=2^{n-2}+m''$ and $ n-2>t'_{0}>t'_{1}>\cdot\cdot\cdot>t'_{s}$. By  the equation~\eqref{formula5} and Lemma~\ref{s8},
\begin{align*}
\xi_{m}(Q_{n,2})&-\xi_{2^{n-3}}(Q_{n,2}) \\
&=\xi_{m}(Q_{n,2})-\xi_{2^{n-2}}(Q_{n,2})+\xi_{2^{n-2}}(Q_{n,2})-\xi_{2^{n-3}}(Q_{n,2})\\
&=(n+1)(2^{n-2}+m'')-(n+1)2^{n-2}-(e x_{m}(Q_{n,2})-e x_{2^{n-2}}(Q_{n,2}))+2^{n-2}\\
&=(n+1)m''-(ex_{2^{n-2}+m''}(Q_{n})+2m'')+ex_{2^{n-2}}(Q_{n})+2^{n-2}~~(\textrm{Lemma~\ref{s3}})\\
&=(n+1)m''-ex_{2^{n-2}}(Q_{n})-ex_{m''}(Q_{n})-4m''+ex_{2^{n-2}}(Q_{n})+2^{n-2}\\
&=(n+1)m''-ex_{m''}\left(Q_{n}\right)-4m''+2^{n-2}=(n-3)m''-ex_{m''}\left(Q_{n}\right)+2^{n-2}\\
&=(n-3)m''-ex_{m''}\left(Q_{n-3}\right)+2^{n-2}
=\xi_{m''}(Q_{n-3})+2^{n-2}
>0.
\end{align*}

For $0< m'' \leq 2^{n-2},$ the
value of $ex_{m''}(Q_{n})$ is uniquely determined by the binary representation of $m''$. Thus, $ex_{m''}(Q_{n})=ex_{m''}(Q_{n-3})$. By Lemma~\ref{s20}, $(n-3)m''-ex_m''\left(Q_{n-3}\right)>0$ for $0< m'' \leq 2^{n-2}.$ Thus, $\xi_{m}(Q_{n,2})>\xi_{m_{n,r}}(Q_{n,2})$.

Combining the above two cases, $\xi_{m}(Q_{n,2})> \xi_{m_{n,r}}(Q_{n,2})=\xi_{m_{n,r+1}}(Q_{n,2})=\xi_{\left\lceil\frac{11\times2^{n-1}}{48}\right\rceil}(Q_{n,2})=2^{n-1}$ for $1\leq r\leq \lceil\frac{n}{2}\rceil-2$. So the proof is completed.
\end{proof}
\section{The $h$-extra edge-connectivity of $Q_{n,2}$ concentrates on $2^{n-1}$ \\for $\left\lceil\frac{11\times2^{n-1}}{48}\right\rceil\leq h \leq 2^{n-1}$}\label{sec4}

\text{\textbf{The proof of Theorem }} \ref{c1} \text {(a):} \\
Given each integer $h$, for $\left\lceil\frac{11\times2^{n-1}}{48}\right\rceil\leq h \leq m_{n,\lceil\frac{n}{2}\rceil-1}=2^{n-1}$, there exists an integer $r, 1\leq r\leq \lceil\frac{n}{2}\rceil-1,$ satisfying $m_{n,r}\leq h\leq m_{n,r+1}$.
By Lemma~\ref{s5} and Lemma~\ref{s10},
$$\lambda_h(Q_{n,2})=\min\{\xi_m(Q_{n,2}): m_{n,r} \leq h \leq m < m_{n,r+1}\}=\xi_{m_{n,r}}(Q_{n,2})$$
for $r=1,2, \cdots, \lceil\frac{n}{2}\rceil-1$. So
for any $\left\lceil\frac{11\times2^{n-1}}{48}\right\rceil\leq h \leq 2^{n-1}$,
\begin{align*}
\lambda_{h}\left(Q_{n,2}\right)&=\min\{\xi_{m}(Q_{n,2}): h\leq m\leq 2^{n-1}\}  \qquad\textrm{(Lemma~\ref{s5})}\\
&=\min\{\xi_{m}(Q_{n,2}): h\leq m\leq m_{n,\lceil\frac{n}{2}\rceil-1}\}    \qquad\textrm{(Lemmas~\ref{s10} and~\ref{s15})}\\
&=\xi_{2^{n-1}}(Q_{n,2})  \qquad\textrm{(Lemmas~\ref{s9} and~\ref{s10})}\\
&=2^{n-1}.
\end{align*}

\noindent
\text {\textbf{The proof of Theorem }} \ref{c1}\text{(b):}\\
If $h = m_{n,r}$ or $h = m_{n,r+1},$ by Lemma~\ref{s5} and Lemma~\ref{s9}} $\lambda_{h}(Q_{n,2})=\xi_{h}(Q_{n,2})=2^{n-1}$. If $m_{n,r} < h < m_{n,r+1}$, by Lemma~\ref{s10}, $\xi_{h}(Q_{n,2})>\xi_{m_{n,r+1}}(Q_{n,2})$, by Lemma~\ref{s5} and Lemma~\ref{s9},
\begin{align*}
\lambda_{h}(Q_{n,2})&=\min\{\xi_m(Q_{n,2}): h \leq m \leq m_{n,r+1}\}\\
&=\xi_{m_{n, r+1}}(Q_{n,2})=\xi_{m_{n, r+2}}(Q_{n,2})=\cdot\cdot\cdot=\xi_{m_{n, \lceil\frac{n}{2}\rceil-1}}(Q_{n,2})=2^{n-1}.
\end{align*}
So, one can get $\lambda_h(Q_{n, 2})=\xi_h(Q_{n, 2})=2^{n-1}$ for $h=m_{n, r}$ or $h=m_{n, r+1}, 1 \leq r < \lceil\frac{n}{2}\rceil-1$.

 The proof is completed.

\medskip

\begin{remark} For $\left\lceil\frac{11\times2^{n-1}}{48}\right\rceil\leq h \leq 2^{n-1}$ and $n\geq9$, the lower and upper bounds of $h$ of $\lambda_h(Q_{n,2})$ in the above Theorem~\ref{c1} are both tight.

(1) In fact, if $n$ is even, then
\begin{align*}
m_{n,1}&=\textstyle{\sum\nolimits_{i=0}^{2} 2^{n-4-i}+\sum\nolimits_{i=0}^{\lceil\frac{n}{2}\rceil-5} 2^{n-8-2i}+2},\\
m_{n,1}-1&=\textstyle{\sum\nolimits_{i=0}^{2} 2^{n-4-i}+\sum\nolimits_{i=0}^{\lceil\frac{n}{2}\rceil-5} 2^{n-8-2i}+1}.
\end{align*}
By Lemma~\ref{s5}, $ex_{m_{n,1}}(Q_{n,2})=ex_{m_{n,1}-1}(Q_{n,2})+n$. So,
$$\xi_{m_{n,1}}(Q_{n,2})-\xi_{m_{n,1}-1}(Q_{n,2})=(n+1)m_{n,1}-ex_{m_{n,1}}(Q_{n,2})-(n+1)(m_{n,1}-1)+ex_{m_{n,1}-1}(Q_{n,2})=1.$$
Note that
\begin{align*}
\lambda_{m_{n,1}-1}(Q_{n,2})&=\min\{\xi_{h}(Q_{n,2}) : m_{n,1}-1\leq h \leq m_{n,1}\}=\xi_{m_{n,1}-1}(Q_{n,2})\\
&=2^{n-1}-1<2^{n-1}=\lambda_{m_{n,1}}(Q_{n,2})=\xi_{m_{n,1}}(Q_{n,2}).
\end{align*}
If $n$ is odd, then
\begin{align*}
m_{n,1}&=\textstyle{\sum\nolimits_{i=0}^{2} 2^{n-4-i}+\sum\nolimits_{i=0}^{\lceil\frac{n}{2}\rceil-5} 2^{n-8-2i}+1},\\
 m_{n,1}-1&=\textstyle{\sum\nolimits_{i=0}^{2} 2^{n-4-i}+\sum\nolimits_{i=0}^{\lceil\frac{n}{2}\rceil-5} 2^{n-8-2i}}.
 \end{align*}
 By Lemma~\ref{s5}, $ex_{m_{n,1}}(Q_{n,2})=ex_{m_{n,1}-1}(Q_{n,2})+n-1$. So,
 $$\xi_{m_{n,1}}(Q_{n,2})-\xi_{m_{n,1}-1}(Q_{n,2})=(n+1)m_{n,1}-ex_{m_{n,1}}(Q_{n,2})-(n+1)(m_{n,1}-1)+ex_{m_{n,1}-1}(Q_{n,2})=2.$$
Similarly, it can be seen that
\begin{align*}
\lambda_{m_{n,1}-1}(Q_{n,2})&=\min\{\xi_{m}(Q_{n,2}) : m_{n,1}-1\leq m \leq m_{n,1}\}=\xi_{m_{n,1}-1}(Q_{n,2})\\
&=2^{n-1}-2<2^{n-1}=\lambda_{m_{n,1}}(Q_{n,2})=\xi_{m_{n,1}}(Q_{n,2}).
\end{align*}
Therefore, the lower bound is sharp.

(2) As $|V(Q_{n,2})|=2^{n},$  by the definition of $h$-extra edge-connectivity, there are at least two components with at least $h$ vertices. So, the upper bound of the above interval is $2^{n-1}$. Therefore, the upper bound is sharp.
\end{remark}

In cases where $4\leq n\leq 9$ and $h\leq2^{n-1}$, the values of $\lambda_{h}(Q_{n,2})$ and $\xi_{h}(Q_{n,2})$ are listed in Table~\ref{tab4}. And the values of $\lambda_{h}(Q_{n,2})$ do not satisfy the equality $\lambda_{h}(Q_{n,2})=\xi_{h}(Q_{n,2})$ are marked in red, otherwise are marked in black. Based on these data, the scatter plots of $\xi_h(Q_{n,2})$ and $\lambda_h(Q_{n,2})$ are plotted. We plot the $\xi_h(Q_{n,2})$ marked in
 `` $\Delta$ " scatters and the $\lambda_h(Q_{n,2})$  marked in `` $*$ " scatters for $4 \leq n \leq 12 $ in Fig~\ref{fig5}.  On the $X$-axis in Fig.~\ref{fig5}, the results of this article are represented by the green lines.

 \begin{table}[h]
\centering
\caption{Examples of $\xi_h(Q_{n,2})$ and $\lambda_h(Q_{n,2})$ for $4\leq n\leq9$.}\label{tab4}
\resizebox{\textwidth}{!}
{
\begin{tabular}{@{}llllllllllllllllllllllllllllllllllllllllllllllllllllllllllllllllllllllllllll@{}}
\hline
\hline
 $h$  & 1  & 2 & 3 & 4 & 5 & 6 & 7 & 8 & 9 & 10 & 11  & 12  & 13  & 14  & 15  & 16 & 17 & 18 & 19  & 20 & 21  & 22  & 23 & 24 & 25 & 26 & 27 & 28 & 29 & 30 & 31 & 32 \\ \hline
$\xi_h(Q_{4,2})$  & 5   & 8   & 11  & 12  & 13  & 12  & 11  & 8   &     &     &     &     &     &     &     &  & & & & & & & & & & & & & & & &  \\
$\lambda_h(Q_{4,2})$  & 5   & 8   &{\color{red}8}   &{\color{red}8}   &{\color{red}8}   &{\color{red}8}   &{\color{red}8}   & 8   &     &     &     &     &     &     &     &   &&&&&&&&&&&&&&&& \\
$\xi_h(Q_{5,2})$  & 6   & 10  & 14  & 16  & 20  & 22  & 24  & 24  & 26  & 26  & 26  & 24  & 24  & 22  & 20  & 16 &&&&&&&&&&&&&&&&\\
$\lambda_h(Q_{5,2})$  & 6   & 10  & 14  & 16  &{\color{red}16}  &{\color{red}16}  &{\color{red}16}  &{\color{red}16}  &{\color{red}16}  & {\color{red}16}  &{\color{red}16}  &{\color{red}16}  &{\color{red}16}  &{\color{red}16}  &{\color{red}16}  & 16 &&&&&&&&&&&&&&&&\\
$\xi_h(Q_{6,2})$  & 7   & 12  & 17  & 20  & 25  & 28  & 31  & 32  & 37  & 40  & 43  & 44  & 47  & 48  & 49  & 48 & 51  & 52  & 53  & 52  & 53  & 52  & 51  & 48  & 49  & 48  & 47  & 44  & 43  & 40  & 37  & 32 \\
$\lambda_h(Q_{6,2})$   & 7   & 12  & 17  & 20  & 25  & 28  & 31  & 32  &{\color{red}32}  &{\color{red}32}  &{\color{red}32}  &{\color{red} 32}  &{\color{red}32}  &{\color{red}32}  &{\color{red}32}  &{\color{red}32} &{\color{red}32}  &{\color{red}32}  &{\color{red}32}  &{\color{red}32}  &{\color{red}32}  &{\color{red}32}  &{\color{red}32}  &{\color{red}32}  &{\color{red}32}  &{\color{red}32}  &{\color{red}32}  &{\color{red}32}  &{\color{red}32}  &{\color{red} 32}  &{\color{red}32}  & 32\\
$\xi_h(Q_{7,2})$  & 8   & 14  & 20  & 24  & 30  & 34  & 38  & 40  & 46  & 50  & 54  & 56  & 60  & 62  & 64  & 64 & 70  & 74  & 78  & 80  & 84  & 86  & 88  & 88  & 92  & 94  & 96  & 96  & 98  & 98  & 98  & 96 \\
$\lambda_h(Q_{7,2})$   & 8   & 14  & 20  & 24  & 30  & 34  & 38  & 40  & 46  & 50  & 54  & 56  & 60  & 62  & 64  & 64 &{\color{red}64}  &{\color{red}64}  &{\color{red}64}  &{\color{red}64}  &{\color{red}64}  &{\color{red}64}  &{\color{red}64}  &{\color{red}64}  &{\color{red}64}  &{\color{red}64}  &{\color{red}64}  &{\color{red} 64}  &{\color{red}64}  &{\color{red} 64}  &{\color{red} 64}  &{\color{red} 64}\\
$\xi_h(Q_{8,2})$ &9 & 16 & 23   & 28  & 35   & 40  & 45  & 48  &  55  & 60   &  65  & 68  & 73   & 76   & 79   & 80&87   &92   &97   &100   &105   &108    & 111   & 112   & 117  & 120   & 123  & 124  &  127  & 128  &129   &128\\
$\lambda_h(Q_{8,2})$ &9 & 16 & 23   & 28  & 35   & 40  & 45  & 48  &  55  & 60   &  65  & 68  & 73   & 76   & 79   & 80&87   &92   &97   &100   &105   &108 & 111   & 112   & 117  & 120   & 123  & 124  &  127  &{\color{red}128}&{\color{red}128}&{\color{red}128}\\
$\xi_h(Q_{9,2})$ & 10  &  18    &  26     &  32    &  40     &  46  & 52  & 56  & 64  & 70   &  76  & 80   &  86 &  90  & 94   &96&104  & 110   & 116  &120   & 126  & 130  & 134  & 136  &142  &146   & 150  &152   &156   & 158  &160   &160\\
$\lambda_h(Q_{9,2})$ & 10  &  18    &  26     &  32 &  40  &  46  & 52   & 56  & 64  & 70   &  76  & 80   &  86   &  90  & 94   & 96 & 104  & 110   & 116  &120   & 126  & 130  & 134  & 136  & 142  & 146   & 150  & 152   & 156   & 158  & 160   & 160\\
\hline
$h$               & 33  & 34  & 35  & 36  & 37  & 38  & 39  & 40  & 41  & 42  & 43  & 44  & 45  & 46  & 47  & 48&49&50&51&52&53&54&55&56&57&58&59&60&61&62&63&64\\
\hline
$\xi_h(Q_{7,2})$  & 100 & 102 & 104 & 104 & 106 & 106 & 106 & 104 & 106 & 106 & 106 & 104 & 104 & 102 & 100 & 96& 98  & 98  & 98  & 96  & 96  & 94  & 92  & 88  & 88  & 86  & 84  & 80  & 78  & 74  & 70  & 64  \\
$\lambda_h(Q_{7,2})$ &{\color{red} 64}  &{\color{red}64}  &{\color{red}64}  &{\color{red}64}  &{\color{red}64}  &{\color{red}64}  &{\color{red} 64}  &{\color{red}64}  &{\color{red}64}  &{\color{red}64}  &{\color{red}64}  &{\color{red}64}  &{\color{red}64 } &{\color{red} 64}  &{\color{red} 64}  &{\color{red} 64} &{\color{red}64}  &{\color{red}64}  &{\color{red}64}  &{\color{red}64}  &{\color{red}64}  &{\color{red}64}  &{\color{red} 64}  &{\color{red}64}  &{\color{red} 64}  &{\color{red}64 } &{\color{red} 64}  &{\color{red}64}  &{\color{red}64}  &{\color{red} 64}  &{\color{red} 64}  & 64 \\
$\xi_h(Q_{8,2})$ & 135  & 140  & 145  &148   & 153  & 156   & 159  & 160  &165   &  168  & 171  & 172  & 175  & 176  &  177  & 176&181  &184   &187  &188  & 191  & 192  & 193  & 192  &195   & 196 &197  & 196 & 197 & 196  &195   &192\\
$\lambda_h(Q_{8,2})$ &{\color{red}128}&{\color{red}128}&{\color{red}128}&{\color{red}128}&{\color{red}128}&{\color{red}128}&{\color{red}128}&{\color{red}128}
&{\color{red}128}&{\color{red}128}&{\color{red}128}&{\color{red}128}&{\color{red}128}&{\color{red}128}&{\color{red}128}&{\color{red}128}
&{\color{red}128}&{\color{red}128}&{\color{red}128}&{\color{red}128}&{\color{red}128}&{\color{red}128}&{\color{red}128}&{\color{red}128}
&{\color{red}128}&{\color{red}128}&{\color{red}128}&{\color{red}128}&{\color{red}128}&{\color{red}128}&{\color{red}128}&{\color{red}128}\\
$\xi_h(Q_{9,2})$ &168  &174  &180  & 184  & 190  & 194  & 198  & 200  & 206  & 210  &214  &216   &220   & 222  & 224  &224&230 &234 &238 &240 &244 &246 &248 &248 &252 &254 &256 &256 & 258&258 &258 &256\\
$\lambda_h(Q_{9,2})$ &168  &174  &180  & 184  & 190  & 194  & 198  & 200  & 206  & 210  &214  &216   &220   & 222  & 224  &224 &230 &234 &238 &240 &244 &246 &248 &248 &252 &254 &{\color{red}256} &{\color{red}256}&{\color{red}256}&{\color{red}256}&{\color{red}256}&{\color{red}256}\\
\hline
$h$ &65&66&67& 68& 69& 70& 71& 72& 73& 74& 75 &76& 77& 78& 79& 80& 81& 82& 83& 84 &85& 86& 87& 88 &89 &90& 91 &92& 93 &94& 95& 96& \\
\hline
$\xi_h(Q_{8,2})$ &197	&200	&203&	204	&207&	208&	209&	208	&211&	212&	213	&212&	213	&212&	211&	208&	211&	212&	213&	212&	213&	212&	211&	208&	209&	208&	207&	204&	203&	200&	197&	192	\\
$\lambda_h(Q_{8,2})$&{\color{red}128}&{\color{red}128}&{\color{red}128}&{\color{red}128}&{\color{red}128}&{\color{red}128}&{\color{red}128}&{\color{red}128}
&{\color{red}128}&{\color{red}128}&{\color{red}128}&{\color{red}128}&{\color{red}128}&{\color{red}128}&{\color{red}128}&{\color{red}128}
&{\color{red}128}&{\color{red}128}&{\color{red}128}&{\color{red}128}&{\color{red}128}&{\color{red}128}&{\color{red}128}&{\color{red}128}
&{\color{red}128}&{\color{red}128}&{\color{red}128}&{\color{red}128}&{\color{red}128}&{\color{red}128}&{\color{red}128}&{\color{red}128}\\
$\xi_h(Q_{9,2})$ &264&	270&	276&	280&	286&	290&	294&	296&	302&	306&	310	&312&	316&	318&	320	&320	&326	&330&	334&	336	&340&	342&	344&	344&	348&	350&	352	&352&	354&	354&	354&	352&\\
$\lambda_h(Q_{9,2})$ &{\color{red}256}&{\color{red}256}&{\color{red}256}&{\color{red}256}&{\color{red}256}&{\color{red}256}&{\color{red}256}&{\color{red}256}
&{\color{red}256}&{\color{red}256}&{\color{red}256}&{\color{red}256}&{\color{red}256}&{\color{red}256}&{\color{red}256}&{\color{red}256}
&{\color{red}256}&{\color{red}256}&{\color{red}256}&{\color{red}256}&{\color{red}256}&{\color{red}256}&{\color{red}256}&{\color{red}256}
&{\color{red}256}&{\color{red}256}&{\color{red}256}&{\color{red}256}&{\color{red}256}&{\color{red}256}&{\color{red}256}&{\color{red}256}\\
\hline
$h$ &97&98& 99 &100& 101& 102& 103& 104& 105& 106& 107& 108& 109& 110& 111& 112& 113& 114& 115 &116& 117& 118& 119& 120& 121 &122& 123& 124& 125& 126& 127& 128\\
\hline
$\xi_h(Q_{8,2})$ &195&	196	&197&	196	&197&	196	&195&	192&	193	&192&	191&	188&	187&	184	&181&	176&	177&	176&	175	&172&	171&	168	&165&	160&	159&	156&	153&	148&	145&	140&	135	&128\\
$\lambda_h(Q_{8,2})$&{\color{red}128}&{\color{red}128}&{\color{red}128}&{\color{red}128}&{\color{red}128}&{\color{red}128}&{\color{red}128}&{\color{red}128}
&{\color{red}128}&{\color{red}128}&{\color{red}128}&{\color{red}128}&{\color{red}128}&{\color{red}128}&{\color{red}128}&{\color{red}128}
&{\color{red}128}&{\color{red}128}&{\color{red}128}&{\color{red}128}&{\color{red}128}&{\color{red}128}&{\color{red}128}&{\color{red}128}
&{\color{red}128}&{\color{red}128}&{\color{red}128}&{\color{red}128}&{\color{red}128}&{\color{red}128}&{\color{red}128}&{\color{red}128}\\
$\xi_h(Q_{9,2})$&358&	362	&366&	368&	372&	374&	376&	376&	380	&382&	384&	384	&386&	386&	386&	384&	388	&390&	392&	392&	394&	394&	394&	392&	394&	394&	394&392	&392&	390&	388&	384\\
$\lambda_h(Q_{9,2})$&{\color{red}256}&{\color{red}256}&{\color{red}256}&{\color{red}256}&{\color{red}256}&{\color{red}256}&{\color{red}256}&{\color{red}256}
&{\color{red}256}&{\color{red}256}&{\color{red}256}&{\color{red}256}&{\color{red}256}&{\color{red}256}&{\color{red}256}&{\color{red}256}
&{\color{red}256}&{\color{red}256}&{\color{red}256}&{\color{red}256}&{\color{red}256}&{\color{red}256}&{\color{red}256}&{\color{red}256}
&{\color{red}256}&{\color{red}256}&{\color{red}256}&{\color{red}256}&{\color{red}256}&{\color{red}256}&{\color{red}256}&{\color{red}256}\\
\hline
$h$ &129 &130 &131& 132& 133& 134& 135& 136& 137& 138& 139& 140 &141& 142& 143& 144 &145& 146& 147 &148& 149& 150& 151& 152& 153& 154& 155 &156& 157& 158& 159 &160\\
\hline
$\xi_h(Q_{9,2})$&390&	394&	398&	400&	404&	406&	408&	408	&412&	414&	416	&416&	418&	418&	418&	416	&420&	422&	424&	424&	426&	426&	426&	424&	426&	426&	426&	424&	424&	422&	420	&416\\
$\lambda_h(Q_{9,2})$&{\color{red}256}&{\color{red}256}&{\color{red}256}&{\color{red}256}&{\color{red}256}&{\color{red}256}&{\color{red}256}&{\color{red}256}
&{\color{red}256}&{\color{red}256}&{\color{red}256}&{\color{red}256}&{\color{red}256}&{\color{red}256}&{\color{red}256}&{\color{red}256}
&{\color{red}256}&{\color{red}256}&{\color{red}256}&{\color{red}256}&{\color{red}256}&{\color{red}256}&{\color{red}256}&{\color{red}256}
&{\color{red}256}&{\color{red}256}&{\color{red}256}&{\color{red}256}&{\color{red}256}&{\color{red}256}&{\color{red}256}&{\color{red}256}\\
\hline
$h$ &161	&162&	163	&164&	165&	166&	167	&168&	169&	170	&171&	172&	173&	174	&175&	176&	177&	178&	179	&180&	181&	182&	183&	184&	185&	186	&187&	188&	189&	190&	191	&192\\
\hline
$\xi_h(Q_{9,2})$&420	&422&	424&	424&	426&	426&	426&	424&	426	&426&	426&	424	&424&	422	&420&	416&	418	&418&	418&	416	&416&	414&	412&	408&	408&	406&	404&	400&	398&	394&	390&	384\\
$\lambda_h(Q_{9,2})$&{\color{red}256}&{\color{red}256}&{\color{red}256}&{\color{red}256}&{\color{red}256}&{\color{red}256}&{\color{red}256}&{\color{red}256}
&{\color{red}256}&{\color{red}256}&{\color{red}256}&{\color{red}256}&{\color{red}256}&{\color{red}256}&{\color{red}256}&{\color{red}256}
&{\color{red}256}&{\color{red}256}&{\color{red}256}&{\color{red}256}&{\color{red}256}&{\color{red}256}&{\color{red}256}&{\color{red}256}
&{\color{red}256}&{\color{red}256}&{\color{red}256}&{\color{red}256}&{\color{red}256}&{\color{red}256}&{\color{red}256}&{\color{red}256}\\
\hline
$h$&193	&194&	195&	196&	197	&198	&199	&200	&201&	202	&203&	204&	205	&206&	207&	208&	209	&210	&211&	212	&213&	214&	215&	216&	217&	218&	219&	220&	221&	222&	223&	224\\
\hline
$\xi_h(Q_{9,2})$&388	&390	&392&	392	&394&	394&	394&	392	&394&	394&	394	&392&	392&	390&	388&	384&	386	&386&	386	&384&	384&	382	&380&	376&	376&	374	&372&	368	&366&	362&	358&	352\\
$\lambda_h(Q_{9,2})$&{\color{red}256}&{\color{red}256}&{\color{red}256}&{\color{red}256}&{\color{red}256}&{\color{red}256}&{\color{red}256}&{\color{red}256}
&{\color{red}256}&{\color{red}256}&{\color{red}256}&{\color{red}256}&{\color{red}256}&{\color{red}256}&{\color{red}256}&{\color{red}256}
&{\color{red}256}&{\color{red}256}&{\color{red}256}&{\color{red}256}&{\color{red}256}&{\color{red}256}&{\color{red}256}&{\color{red}256}
&{\color{red}256}&{\color{red}256}&{\color{red}256}&{\color{red}256}&{\color{red}256}&{\color{red}256}&{\color{red}256}&{\color{red}256}\\
$h$ &225	&226	&227	&228	&229&	230&	231&	232&	233	&234&	235&	236&	237&	238&	239&	240&	241&	242	&243&	244&	245&	246&	247	&248&	249&	250&	251&	252&	253	&254&	255	&256\\
$\xi_h(Q_{9,2})$&354&	354&	354&	352	&352&	350&	348&	344	&344	&342&	340&	336&	334&	330&	326&	320	&320&	318&	316	&312&	310	&306&	302	&296&	294	&290	&286&	280&	276	&270	&264&	256\\
$\lambda_h(Q_{9,2})$&{\color{red}256}&{\color{red}256}&{\color{red}256}&{\color{red}256}&{\color{red}256}&{\color{red}256}&{\color{red}256}&{\color{red}256}
&{\color{red}256}&{\color{red}256}&{\color{red}256}&{\color{red}256}&{\color{red}256}&{\color{red}256}&{\color{red}256}&{\color{red}256}
&{\color{red}256}&{\color{red}256}&{\color{red}256}&{\color{red}256}&{\color{red}256}&{\color{red}256}&{\color{red}256}&{\color{red}256}
&{\color{red}256}&{\color{red}256}&{\color{red}256}&{\color{red}256}&{\color{red}256}&{\color{red}256}&{\color{red}256}&{\color{red}256}\\
\hline
\hline
\end{tabular}}
\end{table}

\begin{figure}[htbp]
\centering
\begin{minipage}[t]{0.45\textwidth}
\vspace{4pt}
\centerline{
\includegraphics[width=7cm,height=4.5cm]{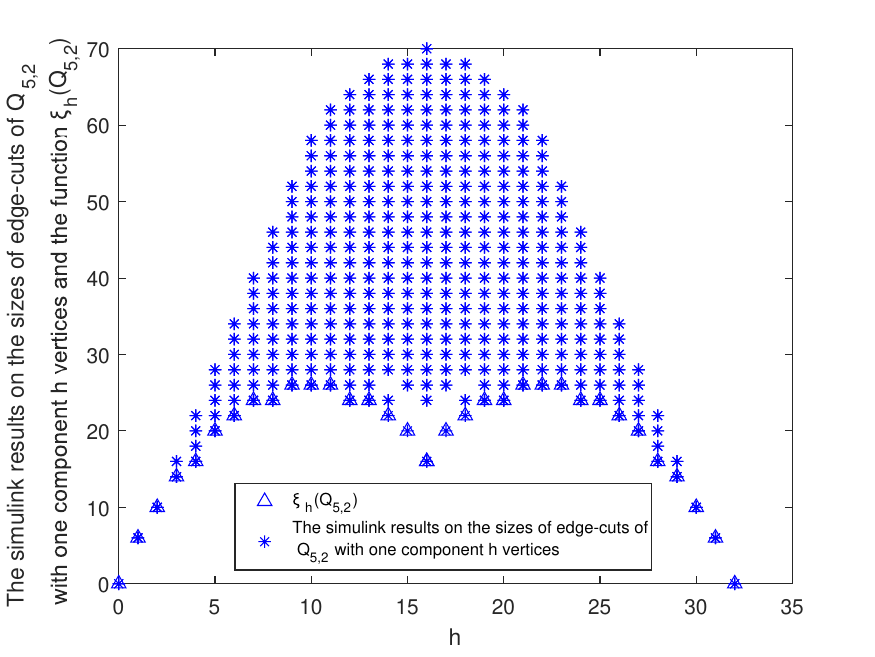}}
\end{minipage}
\begin{minipage}[t]{0.45\textwidth}
\vspace{4pt}
\centerline{
\includegraphics[width=7cm,height=4.5cm]{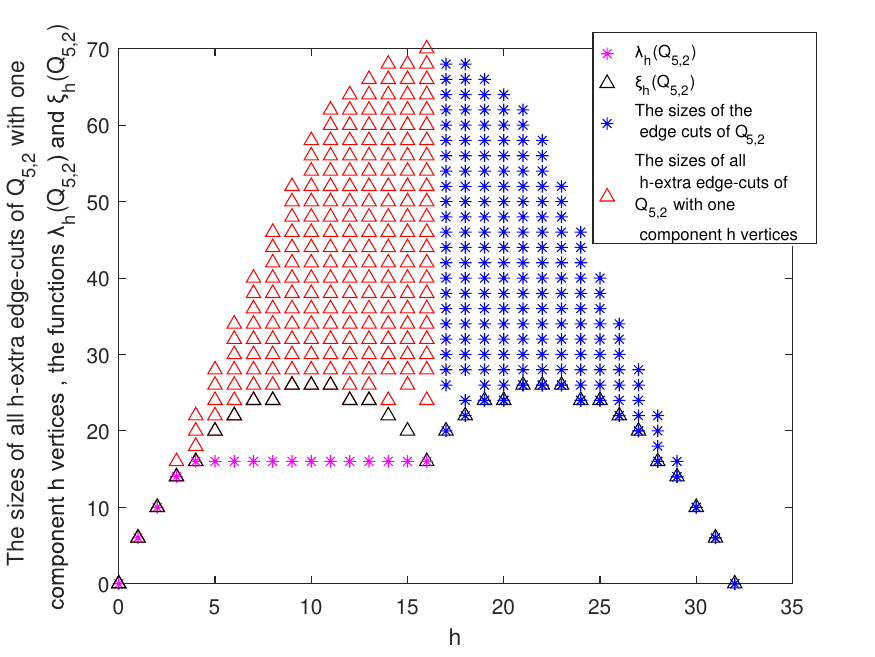}}
\end{minipage}
\caption{The comparison of the sizes of $h$-extra edge-cuts in $Q_{5,2}$ between the simulation and our results.}
\label{fig5}
\end{figure}

\begin{figure}[htbp]
\centering
\begin{minipage}[t]{0.32\textwidth}
\vspace{4pt}
\centerline{
\includegraphics[width=5cm,height=3.2cm]{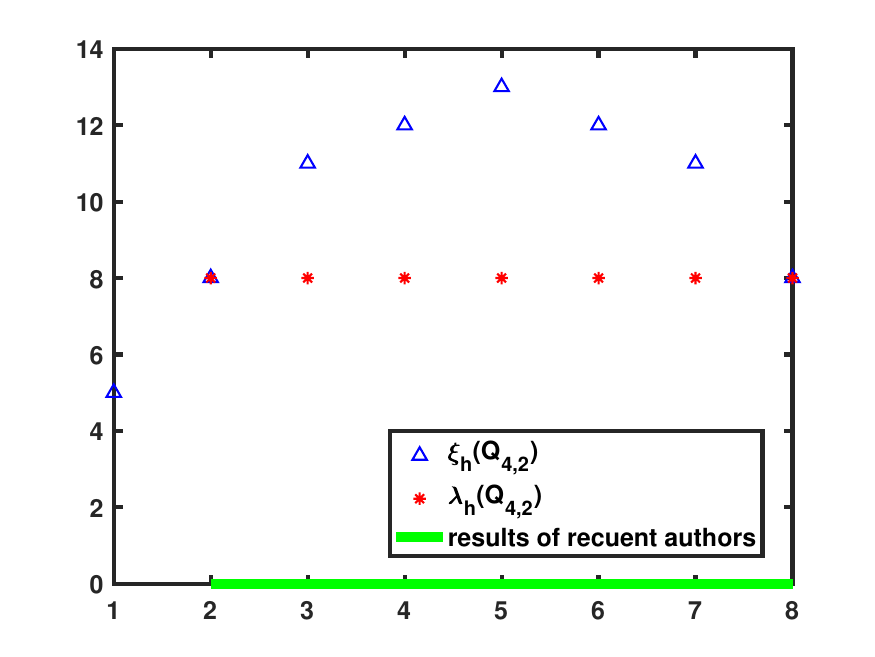}}
\end{minipage}
\begin{minipage}[t]{0.32\textwidth}
\vspace{4pt}
\centerline{
\includegraphics[width=5cm,height=3.2cm]{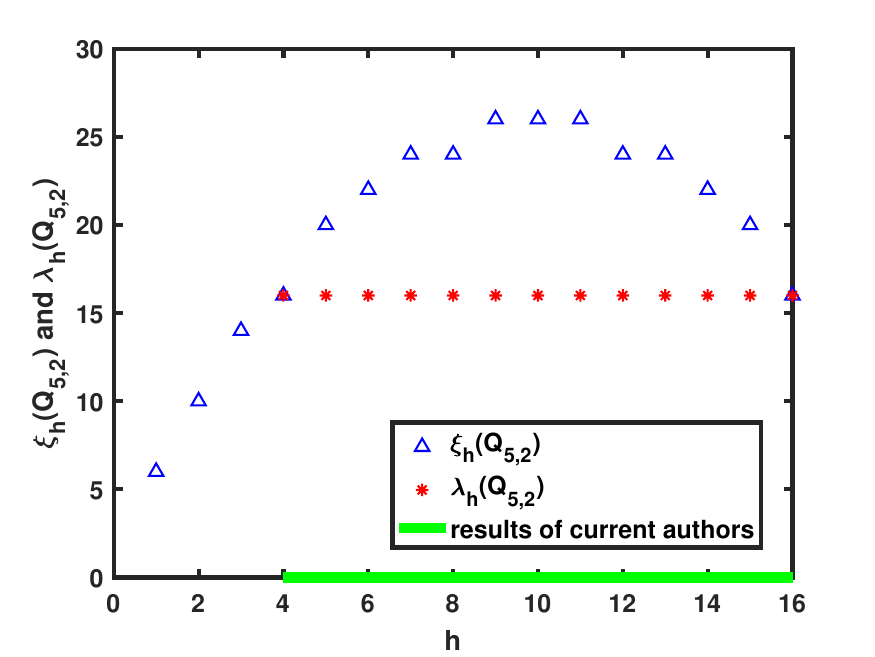}}
\end{minipage}
\begin{minipage}[t]{0.32\textwidth}
\vspace{4pt}
\centerline{
\includegraphics[width=5cm,height=3.2cm]{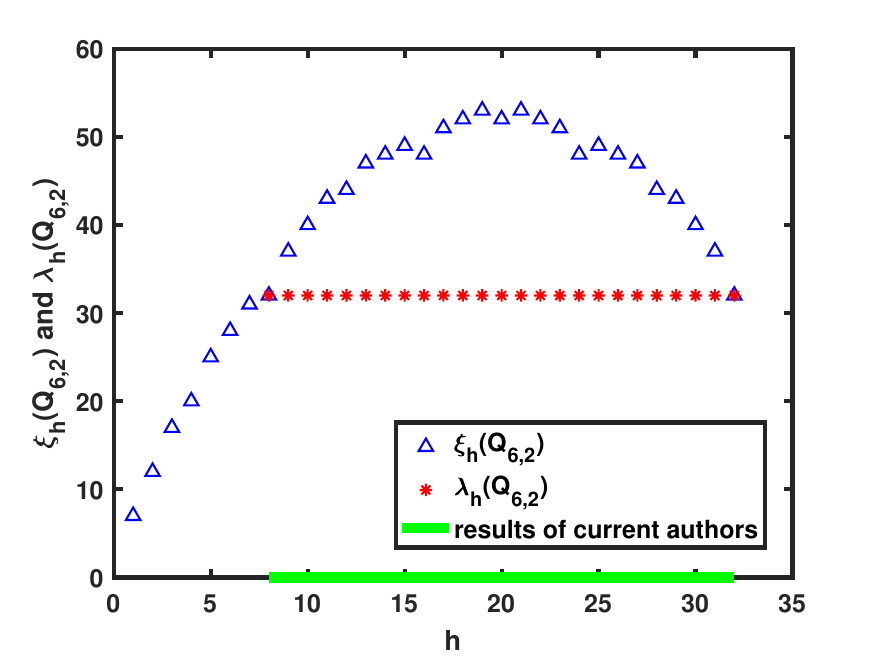}}
\end{minipage}
\begin{minipage}[t]{0.32\textwidth}
\vspace{4pt}
\centerline{
\includegraphics[width=5cm,height=3.2cm]{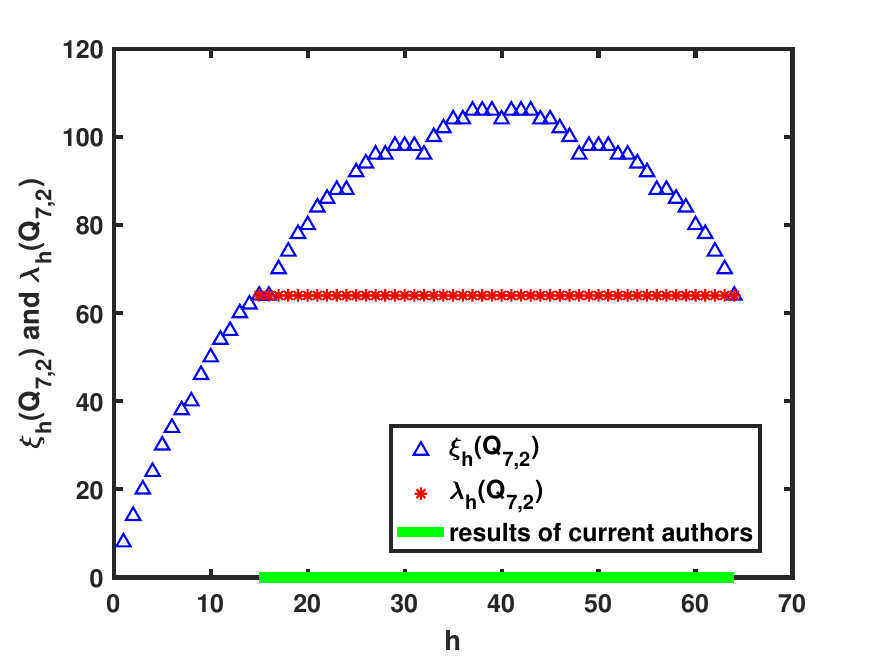}}
\end{minipage}
\begin{minipage}[t]{0.32\textwidth}
\vspace{4pt}
\centerline{
\includegraphics[width=5cm,height=3.2cm]{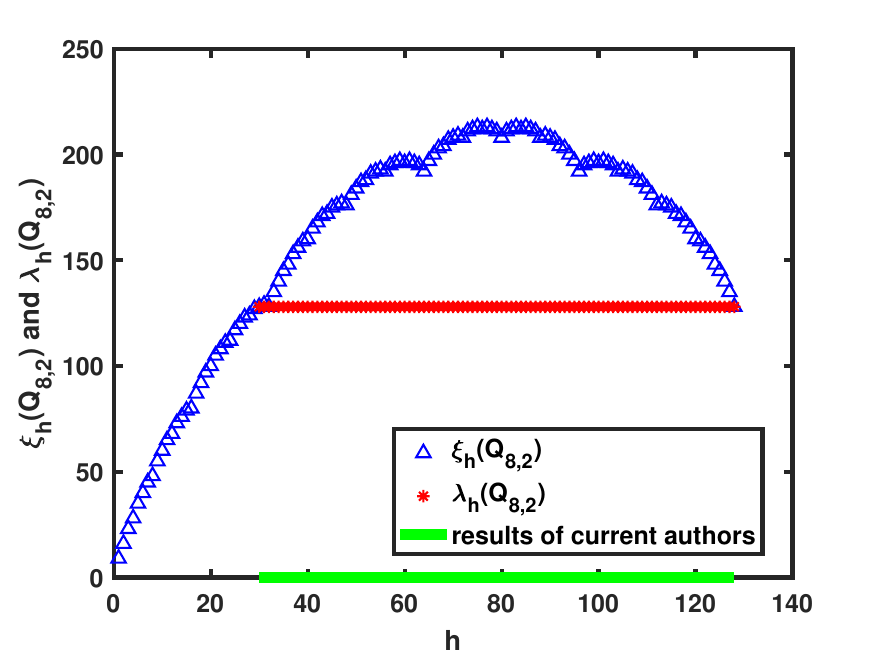}}
\end{minipage}
\begin{minipage}[t]{0.32\textwidth}
\vspace{4pt}
\centerline{
\includegraphics[width=5cm,height=3.2cm]{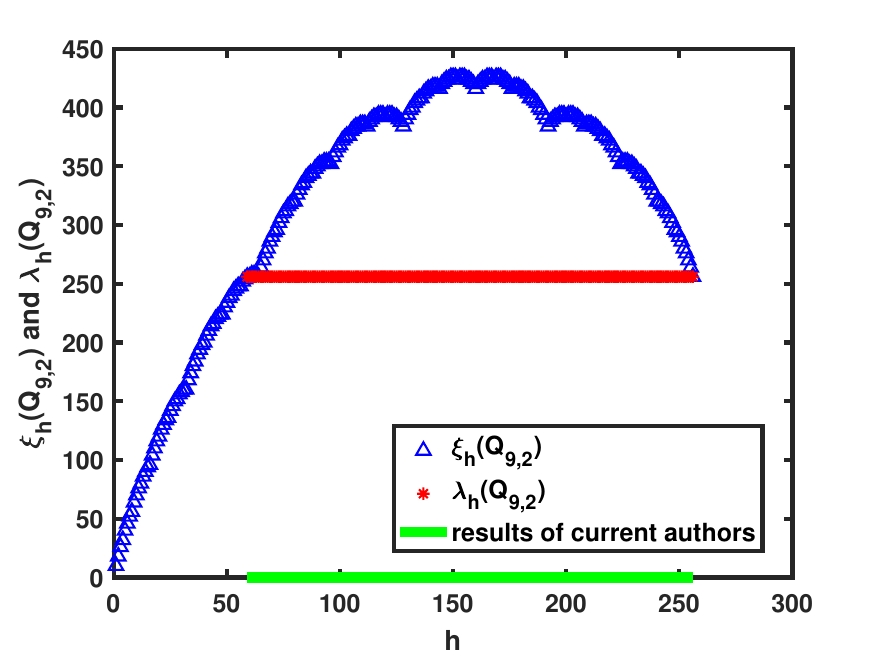}}
\end{minipage}
\begin{minipage}[t]{0.32\textwidth}
\vspace{4pt}
\centerline{
\includegraphics[width=5.3cm,height=3.2cm]{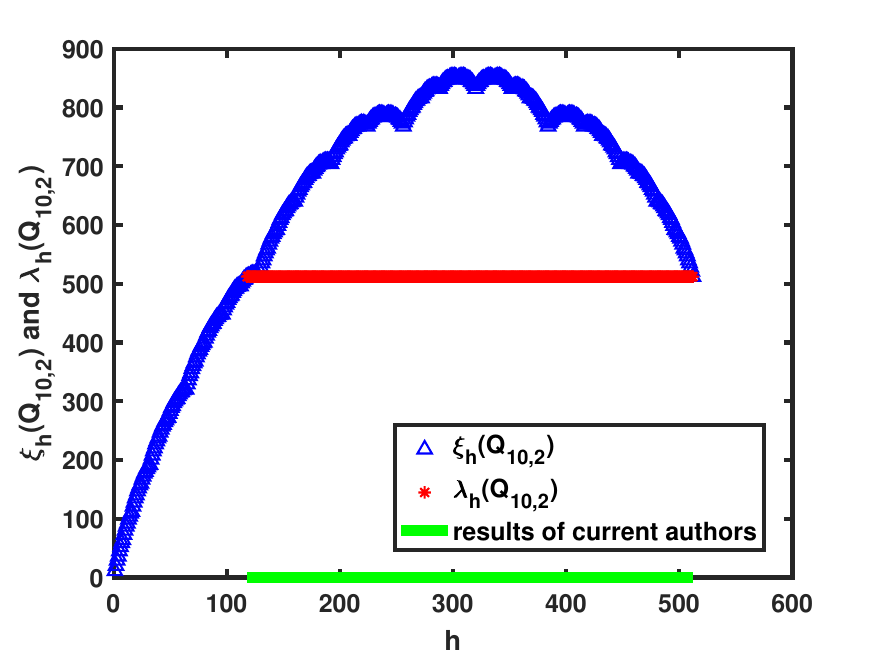}}
\end{minipage}
\begin{minipage}[t]{0.32\textwidth}
\vspace{4pt}
\centerline{
\includegraphics[width=5cm,height=3.2cm]{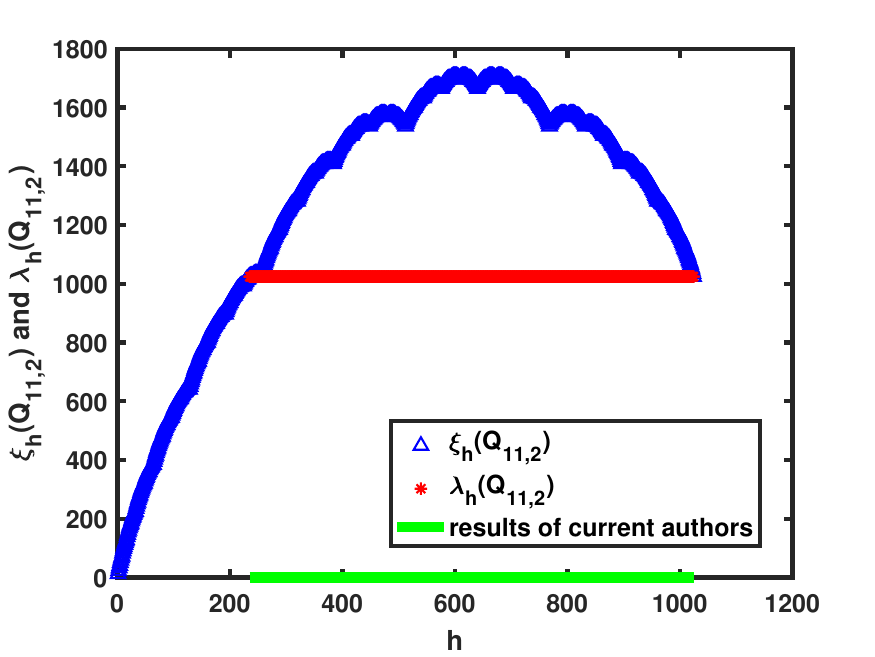}}
\end{minipage}
\begin{minipage}[t]{0.32\textwidth}
\vspace{4pt}
\centerline{
\includegraphics[width=5cm,height=3.2cm]{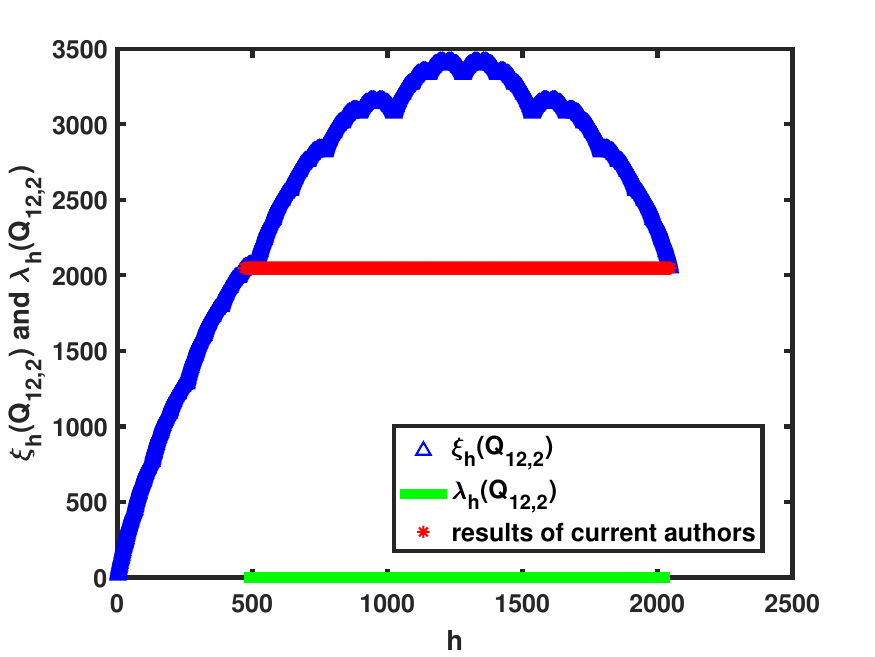}}
\end{minipage}
\caption{The scatter plot of $\lambda_{h}(Q_{n,2})$ and $\xi_{h}(Q_{n,2})$ for case $4\leq n\leq 12$.}
\label{fig6}
\end{figure}

We make a simulation of computing the possible sizes of the edge-cuts
of $Q_{n,2}$ for $n=5$. In the first figure of Fig.~\ref{fig6}, the simulink results for the edge-cuts $[X, \overline{X}]_{Q_{5,2}}$ of $Q_{5,2}$ with one component having $h$ vertices and the function $\xi_h\left(Q_{5,2}\right)$ for $1 \leq h \leq 2^{5}$  are displayed. The possible sizes of the edge-cuts $[X, \overline{X}]_{Q_{5,2}}$ of $Q_{5,2}$  for $h=6$ are $22, 24, 26, 28, 30, 32,$ and $34$ according to the distribution of the first figure of Fig.~\ref{fig6}. The lower bound for these values is $\xi_6\left(Q_{5,2}\right)=22$. The scatter plot of the function $\xi_h\left(Q_{5,2}\right.)$  (depicted in blue `` $\triangle$ " scatters) is symmetric with regard to $h=2^{4}$ because $|[X, \overline{X}]_{Q_{5,2}}|=|[\overline{X}, X]_{Q_{5,2}}|$. In general, the theoretical function $\xi_h\left(Q_{5,2}\right)$ lower bounds our simulation on the sizes of all the edge-cuts $[X, \overline{X}]_{Q_{5,2}}$ with one component containing $h$ vertices for each $0 \leq h \leq2^{4}$.

The sizes of the $h$-extra edge-cuts of $Q_{5,2}, \xi_h\left(Q_{5,2}\right)$ and $\lambda_h\left(Q_{5,2}\right)$ for $h \leq 2^{4}$ are shown in the second figure of Fig.~\ref{fig6}. According to Lemma~\ref{s5},
$$\lambda_h\left(Q_{5,2}\right)=\min \{\xi_m(Q_{5,2}): 1 \leq h \leq m \leq 2^{4}\}.$$
We also find that the $h$-extra edge-connectivity of the $(5,2)$-enhanced hypercube $Q_{5,2}$ presents a concentration phenomenon on the value 16 for $4 \leq h \leq 16$. The results of the simulation are in consistent with those of theoretical analysis.

Unexpectedly, we find that the $h$-extra edge-connectivity of $Q_{n,2}$ exhibits a concentration phenomenon for some exponentially large $h$ on the interval of $\left\lceil\frac{11\times2^{n-1}}{48}\right\rceil\leq h \leq 2^{n-1}$.
Let
$$g(n)=\left|\{h~:~\lambda_h(Q_{n,2})=2^{n-1},  \quad h\leq 2^{n-1}\}\right|.$$
So $g(n)= 2^{n-1}-\lceil\frac{11\times2^{n-1}}{48}\rceil+1$. Due to $|V(Q_{n,2})|=2^{n}, \lambda_h(Q_{n,2})$ is well-defined for any integer $1\leq h\leq 2^{n-1}$. Let $R(n)=\frac{g(n)}{2^{n-1}}$ be the percentage  of the number of integer $h$ with the corresponding $\lambda_h(Q_{n,2})=\xi_h(Q_{n,2})=2^{n-1}$ for $1\leq h\leq 2^{n-1}$.  For the sake of simplicity, Table~\ref{tab5} lists some exact values of the function $R(n)$ for $4\leq n\leq31$. Then $R(n)=\frac{2^{n-1}-\lceil\frac{11\times2^{n-1}}{48}\rceil+1}{2^{n-1}},$  ${\lim\limits_{n \rightarrow \infty}}R(n)=\frac{37}{48}$. The function
$R(n)$ is shown in Fig.~\ref{fig7}. The ratio of the length of the $\lambda_h(Q_{n,2})= 2^{n-1}$ subinterval to the $0\leq h\leq 2^{n-1}$ interval gets infinitely closer to $\frac{37}{48}$ as $n$ grows.  For $n \rightarrow\infty, 77.083\%$ of $\lambda_h(Q_{n,2})$ is $2^{n-1},$ which shows the concentration phenomenon of $\lambda_h(Q_{n,2})$. Furthermore, similar results can be obtained, even if the lower bound of $h$ is not $\left\lceil\frac{11\times2^{n-1}}{48}\right\rceil$ for $4\leq n\leq 8$.\\

\makeatletter\def\@captype{table}\makeatother
\begin{minipage}{.5\textwidth}
\centering
\caption{The values $g(n)$ and $R(n)$ for $4 \leq n \leq 31$.}\label{tab5}
\resizebox{\textwidth}{!}
{
\begin{tabular}{llllllllllllllllll}
\hline
$n$	                &$g(n)$       &$R$ & $n$  &$g(n)$       &$R$\\
\hline
$4$                           &$7$          &$87.5\%$       &$18$                        &$101035$         &$77.083587\%$  \\
$5$                            &$13$          &$81.25\%$   &$19$                    &$202070$        &$77.083587\%$   \\
$6$                            &$25$         &$78.125\%$   &$20$                         &$404139$   &$77.083396\%$\\
$7$                         &$50$         &$78.125\%$     &$21$                     &$808278$ &$77.083396\%$\\
$8$                          &$99$         &$77.34375\%$  &$22$            &$1616555$  &$77.083349\%$\\
$9$      &$198$  &$77.34375\%$   &$23$        &$3233110$ &$77.083349\%$\\
$10$      &$395$  &$77.148437\%$   &$24$        &$6466219$ &$77.083337\%$\\
$11$      &$790$  &$77.148437\%$   &$25$        &$12932438$ &$77.083337\%$\\
$12$      &$1579$  &$77.099609\%$   &$26$        &$25864875$ &$77.083334\%$\\
$13$      &$3158$  &$77.099609\%$   &$27$        &$51729750$ &$77.083334\%$\\
$14$      &$6315$  &$77.087402\%$   &$28$        &$103459499$ &$77.083333\%$\\
$15$      &$12630$  &$77.087402\%$   &$29$        &$206918998$ &$77.083333\%$\\
$16$      &$25259$  &$77.084350\%$   &$30$        &$413837995$ &$77.083333\%$\\
$17$                &$50518$        &$77.084350\%$   &$31$        &$827675990$ &$77.083333\%$\\
\hline
\end{tabular}}
\end{minipage}
\makeatletter\def\@captype{figure}\makeatother
\begin{minipage}{0.45\textwidth}
\centering
\includegraphics[width=6cm,height=5cm]{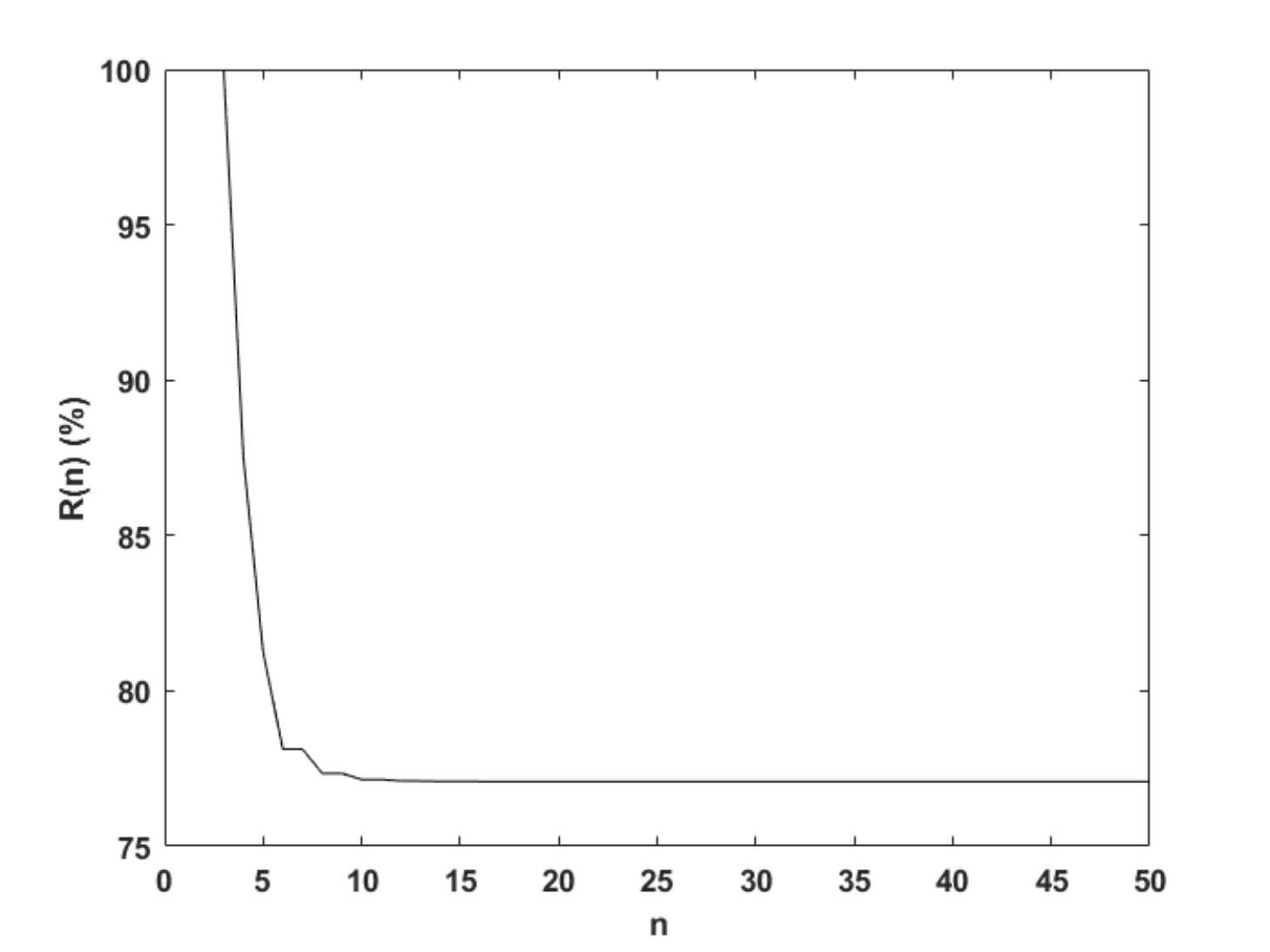}
\caption{The plot of the function $R(n)$.}\label{fig7}
\end{minipage}\\

\section{Conclusion}
It is well known that the $h$-extra edge-connectivity is an important indicator for measuring the fault tolerance and reliability of interconnection networks. This paper shows that the $h$-extra
edge-connectivity of $(n,2)$-enhanced hypercubes $Q_{n,2}$ presents a concentration phenomenon in the subinterval
$$\textstyle{\left\lceil\frac{11\times2^{n-1}}{48}\right\rceil\leq h \leq 2^{n-1}}$$
 for $n\geq9$. For approximately $77.083\%$ values of $h \leq 2^{n-1}$, the minimum number of link malfunctions is $2^{n-1}$, and these link malfunctions
disconnect $(n,2)$-enhanced hypercube $Q_{n,2}$ and keep each resulting connected subnetworks with at least $h$ processors. Our results provide a more accurate measure for evaluating the reliability and availability of large-scale $Q_{n,2}$ networks. In order to completely solve the $h$-extra edge-connectivity of the remaining intervals, we will further investigate an algorithm to determine the exact value and the optimality of the $h$-extra edge-connectivity of $Q_{n,2}$ for each integer $h\leq 2^{n-1}$ in the future. Additionally, for the general network $Q_{n,k}$, we propose to design an algorithm to determine the exact value and the optimality of $\lambda_h(Q_{n,k})$.\\

\noindent \textbf{Acknowledgements}\vspace{1em}
\\ The authors would like to thank referees and editors for their help. Their valuable comments and suggestions help to improve the quality of this paper. \\

%


\bibliography{sn-bibliography}

\begin{thebibliography}{23}
\bibitem{ref31}
M. Arockiaraj, J.-B. Liu, A.J. Shalini, Vertex decomposition method for wirelength problem and its applications to enhanced hypercube networks,
IET Comput. Digit. Tech. 13 (2019) 87-92.
\bibitem{ref3}
N.-W. Chang, C.-Y. Tsai, S.-Y. Hsieh, On $3$-extra connectivity and $3$-extra edge connectivity of folded hypercubes, IEEE Trans. Comput. 63 (6) (2014) 1594-1600.
\bibitem{ref41}
Y.C. Chen, J.M. Tan, L.-H. Hsu, S.-S. Kao, Super-connectivity and super edge-connectivity for some interconnection networks, Appl. Math. Comput. 140 (2003) 245-254.
\bibitem{ref32}
A. El-Amawy, S. Latifi, Properties and performance of folded hypercubes, IEEE Trans. Parallel Distrib. Syst. 2 (1) (1991) 31-42.
\bibitem{ref13}
J. F\`{a}brega, M.A. Fiol, On the extraconnectivity of graphs, Discrete Math. 155 (1996) 49-57.
\bibitem{GE2023}
L.T. Guo, G.B. Ekinci, Connectivity and super connectivity of folded hypercube-like networks, Theor. Comput. Sci. 976 (2023) 114151.
\bibitem{ref8}
F. Harary, Conditional connectivity, Networks 13 (3) (1983) 347-357.
\bibitem{h1964}
L.H. Harper, Optimal assignments of numbers to vertices, J. Soc. Ind. Appl. Math. 12 (1) (1964) 131-135.
\bibitem{ref45}
W.-S. Hong, S.-Y. Hsieh, Extra edge connectivity of hypercube-like networks, Int. J. Parallel Emergent Distrib. Syst. 28 (2) (2013)
123-133.
\bibitem{ref25}
J. Li, Y.Z. Huang, L.M. Lin, H. Yu, R.Q. Chen, The extra connectivity of enhanced hypercubes, Int. J. Parallel Emergent Distrib. Syst. 35 (1) (2020) 91-102.
\bibitem{ref28}
H. Li, W.H. Yang, Bounding the size of the subgraph induced by $m$ vertices and extra edge-connectivity of hypercubes, Discrete Appl. Math. 161 (2013) 2753-2757.
\bibitem{LZHL2024}
H.X. Liu, M.Z. Zhang, S.-Y. Hsieh, and C.-W. Lee, A novel links fault tolerant analysis: $g$-good $r$-component edge-connectivity of interconnection networks with applications to hypercubes, IEEE~Trans.~Reliab., 74 (1) (2025) 2487-2496.
\bibitem{LFCYJ2024}
M.J. Lv, J.X. Fan, B.L. Cheng, J. Yu, X.J. Jia, Construction algorithms of fault-tolerant paths and disjoint paths in $k$-ary $n$-cube networks, J. Parallel Distrib. Comput. 183 (2024) 104761.
\bibitem{MZMM2021}
W.H. Ma, M.Z. Zhang, J.X. Meng, T.L. Ma, Exponential type of many-to-many edge disjoint paths on ternary $n$-cubes, J. Parallel Distrib. Comput. 158 (2021) 67-79.
\bibitem{QMS2024} H.W. Qiao, J.X. Meng, E.M.J. Sabir, The edge fault-tolerant two-disjoint path covers of Cayley graphs generated by a transposition tree, Discrete~Appl.~ Math., 356 (2024) 174-181.
\bibitem{SFMC2024} E.M.J. Sabir, J.X. Fan, J.X. Meng, B.L. Cheng, Structure fault-tolerant hamiltonian cycle and path embeddings in bipartite $k$-ary $n$-cube networks, IEEE Trans. Reliab. 73(1) (2024) 257-269.
\bibitem{ref24}
E. Sabir, A. Mamut, E. Vumar, The extra connectivity of the enhanced hypercubes, Theoret. Comput. Sci. 799 (2019) 22-31.
\bibitem{s1977}
H. Sullivan, T.R. Bashkow, A large scale, homogeneous, fully distributed parallel machine, I, Proceeding 4th Annual Sympo. Comput. archit. (1977) 105-117.
\bibitem{ref1}
N.-F. Tzeng, S.Z. Wei, Enhanced hypercubes, IEEE Trans. Comput. 40 (1991) 284-294.
\bibitem{ref33}
D.J. Wang,  Diagnosability of enhanced hypercubes, IEEE Trans. Comput. 43 (9) (1994) 1054-1061.
\bibitem{ref36}
L.Q. Xu, S.M. Zhou, J.F. Liu, S.S. Yin, Reliability measure of multiprocessor system based on enhanced hypercubes, Discrete Appl. Math. 289 (2021) 125-138.
\bibitem{ref29}
W.H. Yang, H. Li, On reliability of the folded hypercubes in terms of the extra edge-connectivity, Inform. Sci. 272 (2014) 238-243.
\bibitem{ref12}
W.H. Yang, H.Q. Lin, Reliability evaluation of BC networks in terms of the extra vertex- and edge-connectivity, IEEE Trans. Comput. 63 (10) (2014) 2540-2548.
\bibitem{ym2014}
W.H. Yang, J.X. Meng, Extraconnectivity of folded hypercubes, Ars Comb. 116 (2014) 121-127.
\bibitem{YZM2023}
Y.Y. Yang, M.Z. Zhang, J.X. Meng, Fault tolerance analysis for hamming graphs with large-scale faulty links based on $k$-component edge-connectivity, J. Parallel Distrib. Comput. 173 (2023) 107-114.
\bibitem{YZM2024}
Y.Y. Yang, M.Z. Zhang, J.X. Meng, Link fault tolerance of BC networks and folded hypercubes on $h$-extra r-component edge-connectivity, Appl.~ Math.~Comput., 462 (2024) 128343.
\bibitem{YZM2025}
Y.Y Yang, M.Z. Zhang, J.X. Meng, Fault tolerance assessment for hamming graphs based on $r$-restricted $R$-structure (substructure) fault pattern, Appl. Math. Comput., 489 (2025) 129160.
\bibitem{ref62}
S.S. Yin, L.Q. Xu, On the $g$-extra connectivity of the enhanced hypercubes, Comput. J. 65 (9) (2022) 2339-2346.
\bibitem{ref37}
M.Z. Zhang, Edge isopermetric problem on graphs and the related applications, Ph.D. dissertation, University of Xiamen, Xiamen, China, 2018, 68-77.
\bibitem{ZZCH2024}
H. Zhang, S.M. Zhou, E. Cheng, S.-Y. Hsieh,
Characterization of cyclic diagnosability of regular diagnosable networks, IEEE Trans. Reliab. 73(1) (2024) 270-278.
\bibitem{ref23}
M.Z. Zhang, J.X. Meng, W.H. Yang, Y.Z. Tian, Reliability analysis of bijective connection networks in terms of the extra edge-connectivity, Inform. Sci. 279 (2014) 374-382.
\bibitem{ref27}
M.Z. Zhang, L.Z. Zhang, X. Feng, Reliability measures in relation to the $h$-extra edge-connectivity of folded hypercubes, Theoret. Comput. Sci. 615 (2016) 71-77.
\bibitem{ref22}
M.Z. Zhang, L.Z. Zhang, X. Feng, H.-J. Lai, An $O(log_{2}(N))$ algorithm for reliability evaluation of $h$-extra edge-connectivity of folded hypercubes, IEEE Trans. Reliab. 67 (1) (2018) 297-307.
\bibitem{ZZ2024}
H. Zhang, S.M. Zhou, Characterization of matroidal connectivity of regular networks, J. Parallel Distrib. Comput. 186 (2024) 104818.
\bibitem{ref26}
S.L. Zhao, W.H. Yang, Conditional connectivity of folded hypercubes, Discrete Appl. Math. 257 (2019) 388-392.
\bibitem{Z2017}
J.-X. Zhou, On $g$-extra connectivity of hypercube-like networks, J. Comput. Syst. Sci.
88 (2017) 208-219.
\bibitem{ref4}
Q. Zhu, J.-M. Xu, On restricted edge connectivity and extra edge connectivity of hypercubes and folded hypercubes, J. Univ. Sci. Technol. China 36 (2006) 246-253.
\bibitem{ref40}
Q. Zhu, J.-M. Xu, X.M. Hou, M. Xu, On reliability of the folded hypercubes, Inform. Sci. 177 (2007) 1782-1788.
\bibitem{ref55}
Q. Zhu, J.-M. Xu, M. Lv, Edge fault tolerance analysis of a class of interconnection networks, Appl. Math. Comput. 172 (2006) 111-121.
\end{thebibliography}

\end{document}